\documentclass{article}

\usepackage{amsmath}
\usepackage{amssymb}
\usepackage{amsthm}
\usepackage{graphicx}
\usepackage{enumitem}
\usepackage{booktabs}
\usepackage{xcolor}
\usepackage{hyperref}
\definecolor{Mycolor}{RGB}{10,74,38}
\hypersetup{
  colorlinks=true,
  linkcolor=Mycolor,
  citecolor=Mycolor,
  urlcolor=Mycolor
}

\usepackage[ruled,linesnumbered,longend]{algorithm2e}
\usepackage[a4paper,left=2.8cm,right=2.8cm,top=2.5cm,bottom=2.5cm]{geometry}
\usepackage{fancyhdr}
\newtheorem{theorem}{Theorem}[section]
\newtheorem{corollary}{Corollary}[section]
\newtheorem{proposition}{Proposition}[section]

\newtheorem{lemma}{Lemma}[section]
\newtheorem{definition}{Definition}[section]

\makeatletter 
\@addtoreset{equation}{section}
\makeatother  

\newtheorem{remark}{Remark}[section]

\newcommand{\R}{\mathbb{R}}
\newcommand{\Xc}{{\boldsymbol{X}}} 
\newcommand{\dist}{\operatorname{dist}}
\newcommand{\sepa}{\operatorname{sep}}
\newcommand{\supp}{\operatorname{supp}}
\newcommand{\Mc}{\mathcal{M}}
\newcommand{\Pd}{\mathbb{P}}

\begin{document}
\title{Nearly tight framelet systems from nested Marcinkiewicz--Zygmund measures on compact Riemannian manifolds
}

\author{ Hao-Ning Wu\footnotemark[2] \qquad Xiaosheng Zhuang\footnotemark[3]
       }

\renewcommand{\thefootnote}{\fnsymbol{footnote}}
\footnotetext[2]{School of Mathematical Sciences, Xiamen University, Xiamen, Fujian, 361005, China (hnwu@connect.hku.hk)}
\footnotetext[3]{Department of Mathematics, City University of Hong Kong, Hong Kong, China (xzhuang7@cityu.edu.hk)}

\maketitle

\begin{abstract}
Tight framelet systems on manifolds provide exact energy preservation and one-pass reconstruction, but their standard semi-discretization relies on polynomial-exact quadrature rules, which are difficult to reconcile with scattered and progressively refined data.  We develop nested Marcinkiewicz--Zygmund (MZ) measures on compact Riemannian manifolds as a quantitative alternative.   Using dyadic quasi-uniform nested point sets and local partition weights, we establish a sequence of deterministic MZ inequalities for diffusion polynomial spaces.  The result has two complementary forms: a fixed-tolerance version, in which the polynomial bandwidth grows dyadically with the level, and a fixed-bandwidth refinement version, in which the MZ tolerance improves as more nested nodes are added.  When these measures are used to discretize continuous
tight framelets, the MZ tolerance transfers directly to the deviation of the frame bounds from one. Thus quadrature exactness is relaxed with a controlled loss of
tightness, and near-tightness improves along the same nested hierarchy. For the fully discrete setting, we develop filter-bank analysis and synthesis procedures, characterize one-pass and canonical reconstruction, and show that the MZ tolerance also controls the conditioning of the frame-operator equation. Fast implementation of filter-bank transforms is also presented. Numerical experiments on the sphere and the flat torus illustrate the resulting multiscale decompositions and reconstruction behavior.
\end{abstract}

\noindent\textbf{Keywords.} Marcinkiewicz--Zygmund;  nested measures; nearly tight frames; framelets; filter banks.

\noindent\textbf{AMS subject classifications.}  42C15, 42C40, 41A55, 41A10, 41A17, 58C40, 94A12.

\section{Introduction}

Multiresolution analysis was originally developed largely for data in
Euclidean spaces, such as one-dimensional signals, two-dimensional
images, and three-dimensional volumetric or video data.  Multiscale
representation systems that capture sparsity in \(\mathbb R^d\), including
wavelets, framelets, curvelets, and shearlets, have therefore been
extensively developed and widely used; see, e.g.,
\cite{MR3273289,MR2257238,MR1150048,MR1256527,MR1162107,MR1614527,MR1085487,MR1644105}.
Many modern data sets, however, are more naturally modeled
on non-Euclidean domains, such as spherical images, directional data, signals on
surfaces, point clouds, and data graphs.  In such settings, multiscale methods can separate coarse geometric features from fine local
details while respecting the geometry of the underlying domain.  This perspective underlies diffusion maps and diffusion wavelets
\cite{CoifmanLafon2006DiffusionMaps,CoifmanMaggioni2006DiffusionWavelets},
spectral graph wavelets
\cite{HammondVandergheynstGribonval2011SpectralGraphWavelets},
wavelet or frame constructions on compact manifolds and spheres
\cite{GellerMayeli2009NearlyTight,NarcowichPetrushevWard2006Localized,wang2020tight}, and so on.
In this paper, we work on compact Riemannian manifolds, which provide a
common spectral setting for these constructions.

When such manifold systems are discretized, the available samples are
not always placed on a prescribed quadrature grid.  In many practical
situations, they may be scattered or added progressively.  A multiscale
method that can keep the old samples and add new ones at finer levels is
then attractive. Previously acquired samples do not have to be
discarded, and the same hierarchy of point sets can serve both sampling
and filter-bank implementation. 

To see where the sampling problem enters, it is useful to recall the
continuous framelet system before discretization.  On a compact
manifold \(\Mc\), let \(\{u_\ell,\lambda_\ell\}_{\ell\geq0}\) be the
Laplace--Beltrami eigen-pairs.
Let $J_0\in\mathbb{N}_0$ be an integer,
and let
\(\Psi=\{\alpha;\beta^1,\ldots,\beta^r\}\) with $r\geq 1$ be a finite collection of
framelet generators.
  From \(\Psi\), one forms the continuous kernels
\[
        \varphi_{J,y}(x)
        =
        \sum_{\ell\geq0}
        \widehat\alpha\!\left(\frac{\lambda_\ell}{2^J}\right)
        \overline{u_\ell(y)}u_\ell(x),
        \qquad
        \psi^k_{j,y}(x)
        =
        \sum_{\ell\geq0}
        \widehat{\beta^k}\!\left(\frac{\lambda_\ell}{2^j}\right)
        \overline{u_\ell(y)}u_\ell(x),
\]
where \(J\geq J_0\) is the coarsest scaling level, \(j\geq J\) indexes the
finer detail levels with \(k=1,\ldots,r\), and \(y\in\Mc\) is the center
variable.  The continuous framelet system is then given by
\begin{equation}\label{equ:intro-CFS}
        \mathrm{CFS}_J(\Psi)
        =
        \{\varphi_{J,y}:y\in\Mc\}
        \cup
        \{\psi^k_{j,y}:y\in\Mc,\ j\geq J,\ k=1,\ldots,r\}.
\end{equation}
Under the usual spectral partition conditions, \(\mathrm{CFS}_J(\Psi)\) is a continuous tight frame.  Tightness is particularly desirable. The framelet coefficients preserve the \(L^2\)-energy exactly, and reconstruction uses the same system without inverting a frame operator.  To obtain a computable system, however, one must replace the continuum of centers \(y\in\Mc\) by finitely many weighted nodes at each scale.  The passage from the continuous system to a semi-discrete one therefore raises the central question of this paper: how much of the tight-frame property can be retained when the center variable is sampled on scattered and progressively refined node sets?

\subsection{Quadrature rules and nested sampling}
In the tight framelet constructions, see, e.g.,
\cite{wang2020tight,xiao2023spherical}, this
discretization of the center variable \(y\) is achieved by
polynomial-exact quadrature rules.  In a typical form, for a finite set
\(\Xc\subset\Mc\) and positive weights \(\{w_x\}_{x\in\Xc}\), one asks
\begin{equation}
\label{eq:intro-exact-quadrature}
        \int_\Mc P(y)\,d\mu(y)
        =
        \sum_{x\in\Xc_j} w_{j,x}P(x),
        \qquad P\in\Pd_{m_j},
        \qquad m_j\asymp 2^j,
\end{equation}
where \(\Pd_m\) denotes the diffusion polynomial space of bandwidth
\(m\).  
In the standard quadrature-based discretization of manifold framelets, the exactness condition \eqref{eq:intro-exact-quadrature} on the relevant product spaces transfers the continuous tight-frame identity to the semi-discrete system.  In particular, it integrates the squared coefficient functions exactly and therefore preserves the Parseval energy identity. On the sphere, for example, a canonical equal-weight realization of such exact rules is given by
spherical \(t\)-designs \cite{an2026survey,delsarte1991geometriae}, which have also
been used in spherical framelet constructions
\cite{xiao2023spherical}.

Although powerful, such exact rules are highly structured, whereas data nodes in applications are often scattered.  For scattered data, it is
natural to replace exact integration by stable discretization of
polynomial norms.  This leads to Marcinkiewicz--Zygmund 
inequalities on compact manifolds
\cite{filbir2011marcinkiewicz,mhaskar2001spherical}. An \(L^p\)
Marcinkiewicz--Zygmund inequality for \(\Pd_n\) has the form
\begin{equation}
\label{eq:intro-single-mz}
        A\|P\|_{L^p(\mu)}^p
        \leq
        \sum_{x\in\Xc} w_x |P(x)|^p
        \leq
        B\|P\|_{L^p(\mu)}^p,
        \qquad P\in\Pd_n ,
\end{equation}
with constants \(0<A\leq B<\infty\). The Marcinkiewicz--Zygmund inequality is not merely a qualitative stability substitute for
exact quadrature.  When \(A=1-\eta\) and \(B=1+\eta\), it preserves the
continuous polynomial norm up to the relative error \(\eta\).  Applied
to the center-variable framelet coefficients and combined with the
continuous tight-frame identity, this gives a semi-discrete frame whose
bounds are \(1-\eta\) and \(1+\eta\).  Thus the loss of exactness leads
to a quantitatively controlled, rather than arbitrary, loss of
tightness.
Marcinkiewicz--Zygmund inequalities have long been used in stable sampling for positive
quadrature and approximation
\cite{zbMATH05558973,zbMATH05676005,MR3746524,MR4780349,zbMATH07227728,MR2475947},
hyperinterpolation
\cite{an2025path,an2022quadrature,an2024bypassing,sloan1995polynomial,wu2026hyperinterpolation},
and numerical methods for partial differential equations and integral equations \cite{MR5025409,Wu2026MZGalerkin,wu2023breaking}.  For a broader
discussion of the role and limitations of exactness in the design of
quadrature rules, we refer the reader to \cite{MR4373871}.  Related
relaxations of spherical needlet discretizations have also been
studied in \cite{MR4783503} recently, where the analysis relies on quasi-Monte
Carlo designs \cite{MR3246811,SloanWomersley2026QMC} rather than on Marcinkiewicz--Zygmund
inequalities.

The Marcinkiewicz--Zygmund point of view provides a controlled relaxation of exact quadrature rules.  For \(p=2\), exact integration of \(|P|^2\) corresponds to the special case \(A=B=1\), whereas sufficiently fine Marcinkiewicz--Zygmund measures allow the constants to approach one while accommodating geometrically controlled scattered nodes.
Such node sets can be characterized geometrically through their mesh norm and separation in the sense that existing nodes with suitable geometric quality may be retained and assigned appropriate positive weights, rather than replaced by a specially constructed exact quadrature rule.

For multiscale constructions, there is a second requirement: the
sampling sets should be compatible across levels.  Ideally they are
\emph{nested} in the sense of
\[
        \Xc_j\subset \Xc_{j+1},\qquad j\geq0.
\]
Nestedness is the spatial analogue of dyadic refinement: a finer level
keeps the old centers and adds new ones.  This is automatic for some
classical one-dimensional rules, such as dyadically refined
Chebyshev--Lobatto points \cite{Trefethen2013ATAP}, but it is much less
automatic for discretization rules on manifolds.  For spherical
\(t\)-designs, this issue has recently been addressed in
\cite{zheng2024nested,zheng2025probability} by showing that prescribed
point sets can be enlarged to spherical designs of higher degree.  These
works show that nesting is a genuine and useful constraint in
polynomial discretization on manifolds. 
Our aim is therefore to construct sampling hierarchies that combine three properties: geometrically controlled scattered nodes, nestedness across scales, and Marcinkiewicz--Zygmund constants close to one.  The first two properties
make the systems compatible with progressively acquired data, while the
third retains the tight-frame identity up to a prescribed error.
Figure \ref{fig:nested-s2-nodes}
illustrates the kind of nested scattered sampling hierarchy.

\begin{figure}[htbp]
\centering
\begin{minipage}[t]{0.3\textwidth}
\centering
\textbf{(a)} \(\Xc_2\) with \(|\Xc_2|=75\)\\\vspace{1em}
\includegraphics[width=\textwidth]{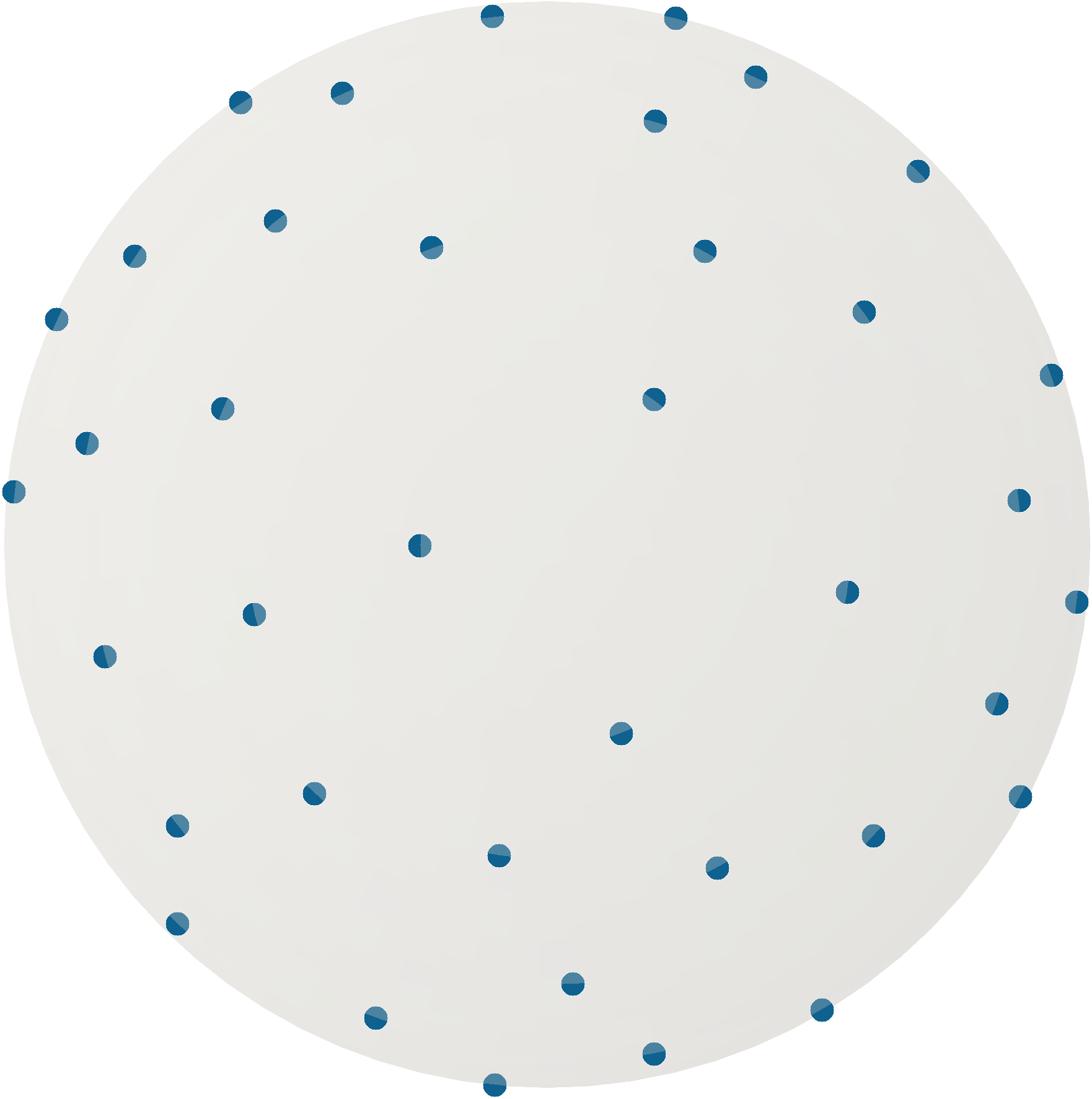}
\end{minipage}
\hfill
\begin{minipage}[t]{0.3\textwidth}
\centering
\textbf{(b)} \(\Xc_3\) with \(|\Xc_3|=243\)\\\vspace{1em} 
\includegraphics[width=\textwidth]{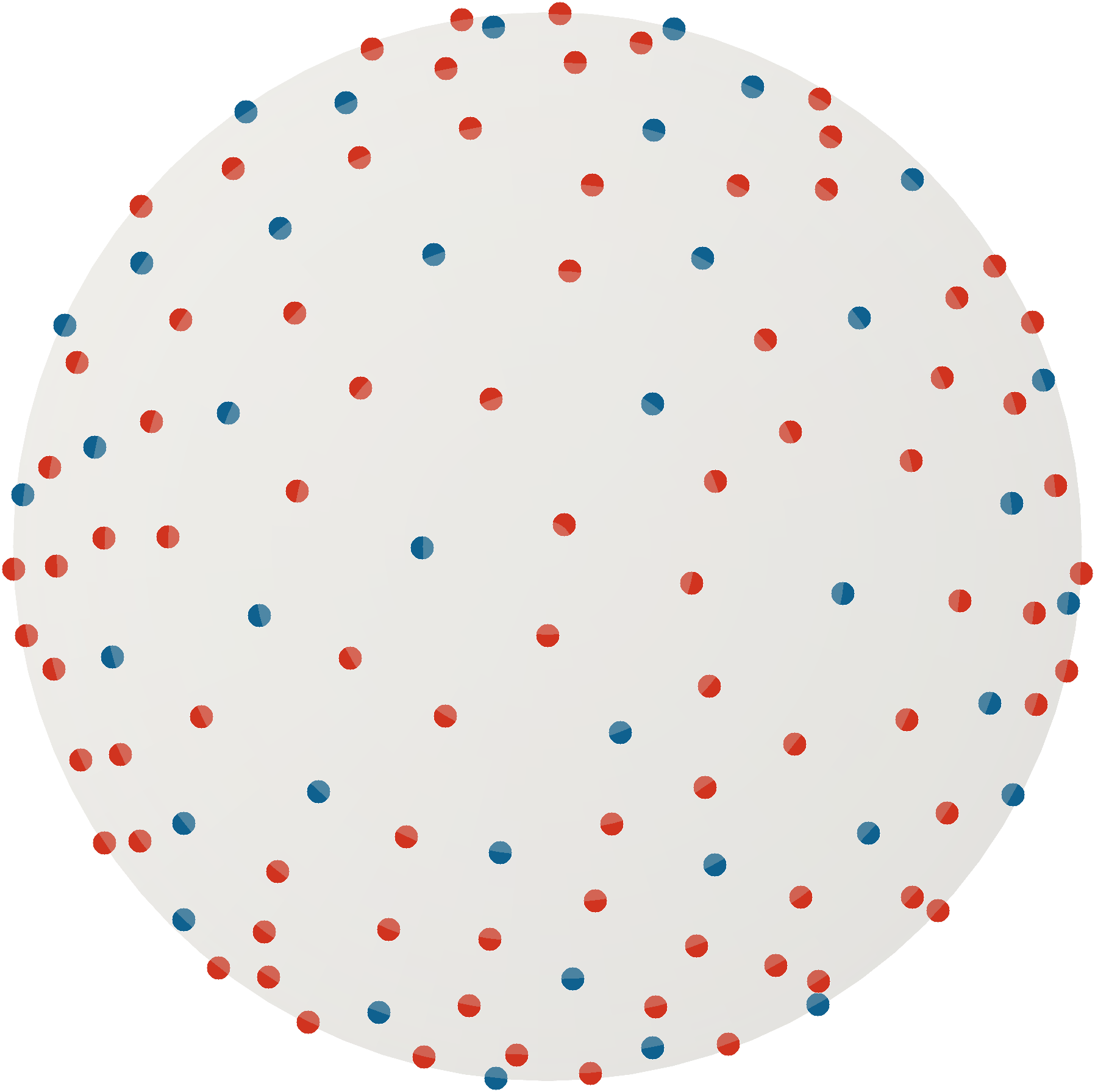}
\end{minipage}
\hfill
\begin{minipage}[t]{0.3\textwidth}
\centering
\textbf{(c)} \(\Xc_4\) with \(|\Xc_4|=867\)\\\vspace{1em}   
\includegraphics[width=\textwidth]{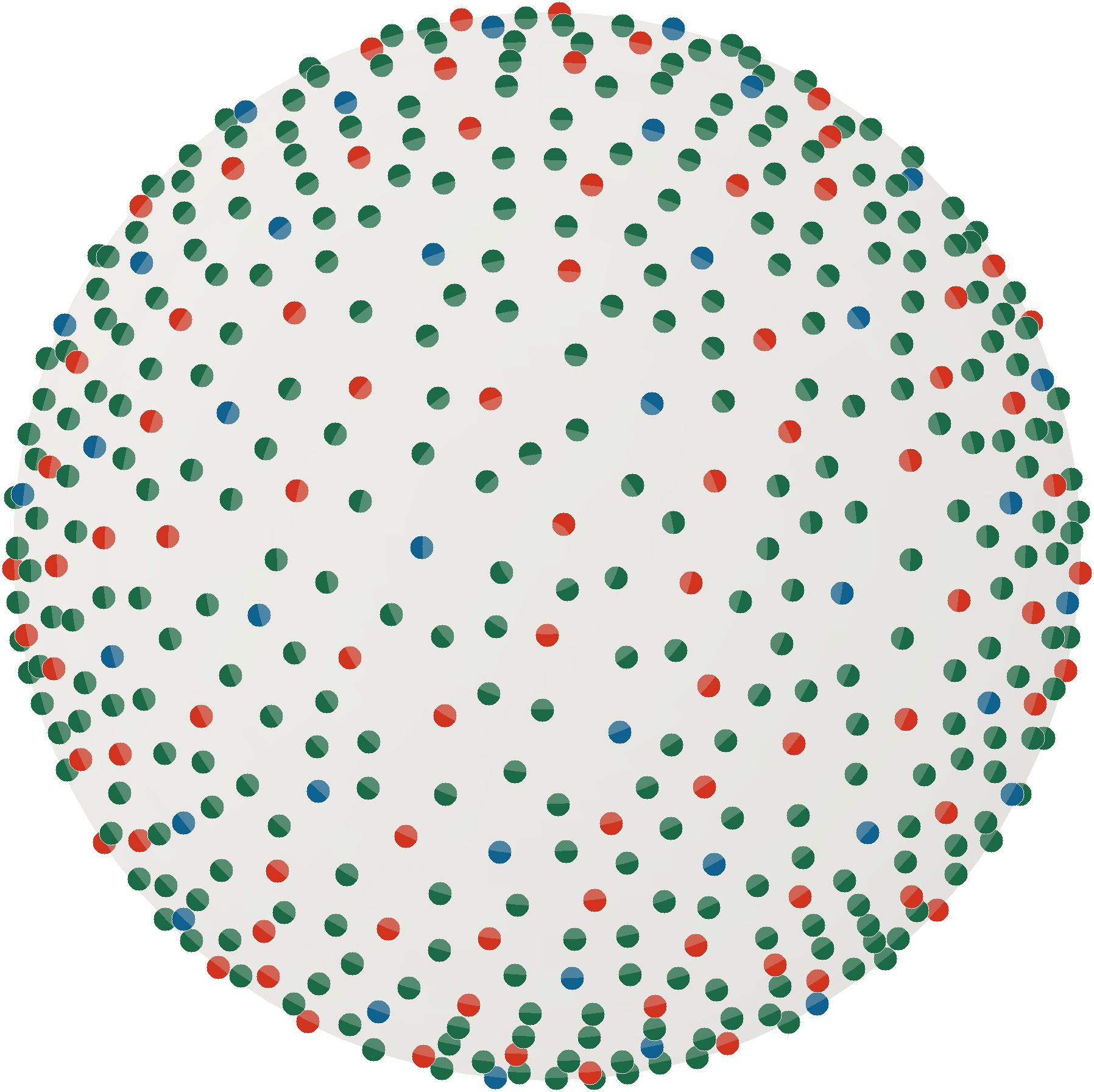}
\end{minipage}
\caption{An illustrative nested scattered sampling family on the sphere.
Blue points form \(\Xc_2\), red points are the new points in
\(\Xc_3\setminus\Xc_2\), and green points are the new points in
\(\Xc_4\setminus\Xc_3\).}
\label{fig:nested-s2-nodes}
\end{figure}

Accordingly, we introduce nested Marcinkiewicz--Zygmund measures.  In the dyadic setting, the
basic goal is to construct nested sets $\Xc_0\subset\Xc_1\subset\cdots$
and positive level-dependent weights such that
\begin{equation}
\label{eq:intro-nested-mz-scale}
(1-\eta)\|P\|_{L^p(\mu)}^p
\leq
\sum_{x\in\Xc_j}w_{j,x}|P(x)|^p
\leq
(1+\eta)\|P\|_{L^p(\mu)}^p,
\qquad P\in\Pd_{n_j},
\qquad n_j\asymp2^j,
\end{equation}
with a fixed \(\eta\) independent of \(j\).  This is the fixed-tolerance
version in the sense that the Marcinkiewicz--Zygmund constants are kept uniformly controlled while the degree
increases with the scale. For \(p=2\), this fixed Marcinkiewicz--Zygmund tolerance becomes exactly the near-tightness
parameter of the resulting framelet system.  We also prove a complementary fixed-degree
version.  If the degree \(n\) is fixed, then the
same construction gives
\begin{equation}
\label{eq:intro-nested-mz-fixed}
        (1-\eta_j)\|P\|_{L^p(\mu)}^p
        \leq
        \sum_{x\in\Xc_j} w_{j,x}|P(x)|^p
        \leq
        (1+\eta_j)\|P\|_{L^p(\mu)}^p,
        \qquad P\in\Pd_n,
\end{equation}
where $\eta_j\to0$ as the nested sets are refined. These two forms reflect two common uses of multiscale sampling: either
one lets the approximation space grow with the dyadic scale, or one oversamples
a fixed space more and more accurately.

\subsection{Nearly tight framelet systems and fully discrete transforms}
The second purpose of the paper is to use nested Marcinkiewicz--Zygmund measures satisfying
\eqref{eq:intro-nested-mz-scale} or \eqref{eq:intro-nested-mz-fixed} in
framelet constructions.  At the semi-discrete level, the measures
discretize the center variable $y$ in the continuous framelet system \eqref{equ:intro-CFS}.  With
the same notation as above, the semi-discrete system is
\[
        \mathrm{FS}_J(\Psi;\Xc,w)
        =
        \{w_{J,x}^{1/2}\varphi_{J,x}:x\in\Xc_J\}
        \cup
        \{w_{j+1,x}^{1/2}\psi^k_{j,x}:
        x\in\Xc_{j+1},\ j\geq J,\ k=1,\ldots,r\},
\]
where \(J\geq J_0\) is the coarsest scaling level, and \(j\geq J\) indexes the
finer detail levels with \(k=1,\ldots,r\).
In the quadrature-based setting, exactness is what preserves the tight
energy identity after discretization.  With Marcinkiewicz--Zygmund sampling, exact tightness is replaced by quantitatively
controlled near-tightness.  A level-independent Marcinkiewicz--Zygmund tolerance \(\eta\)
produces frame bounds \(1-\eta\) and \(1+\eta\).  Moreover, sampling the same coefficient functions on \(s\) additional nested refinement levels reduces the deviation to \(C_{\Mc}2^{-s}\).  Hence the framelet systems become asymptotically tight along the nested hierarchy.

At the fully
discrete level, the same nested measures store the scaling and detail coefficient sequences.  A filter bank
\[
        \Theta=\{a;b^1,\ldots,b^r\}
\]
with low-pass filter \(a\) and high-pass filters \(b^1,\ldots,b^r\) is
linked to the generators $\Psi$ by the refinement equations
\[
        \widehat\alpha(2\xi)=\widehat a(\xi)\widehat\alpha(\xi),
        \qquad
        \widehat{\beta^k}(2\xi)=\widehat {b^k}(\xi)\widehat\alpha(\xi).
\]
On a finite range \(J\leq j\leq K\), after the finest weighted scaling
sequence \(\mathbf v_K\) has been computed on \(\Xc_K\), the analysis
transform is obtained by repeated one-level filter-bank steps
\[
        \mathbf v_{j-1}
        =
        (\mathbf v_j *_j a^*)\downarrow_j,
        \qquad
        \mathbf w_{j-1}^k
        =
        \mathbf v_j *_j (b^k)^*,
        \quad k=1,\ldots,r .
\]
Here \(*_j\) denotes the level-\(j\) spectral convolution and
\(\downarrow_j\) is the downsampling operator, implemented by evaluating the resulting bandlimited function on the coarser weighted node set.
The corresponding direct synthesis, which we call one-pass synthesis, applies the adjoint filter-bank steps to the computed framelet coefficients without solving an additional equation. For a tight system, the frame operator is the identity, so one-pass synthesis recovers the original data exactly. For a Marcinkiewicz--Zygmund-based system, however, the frame bounds generally differ, and one-pass synthesis instead applies the finite-level frame operator. Near-tightness quantifies this departure from exact reconstruction: the frame operator remains close to the identity, while the associated frame-operator equation remains well conditioned. Canonical reconstruction is obtained by inverting this operator, which we implement numerically using conjugate gradients.

\subsection{Contributions}
The main contributions of the paper can be summarized as follows.

\begin{enumerate}[label=\textbf{(\roman*)},leftmargin=2.4em]
\item
We establish nested Marcinkiewicz--Zygmund measures on compact manifolds and give a
deterministic dyadic construction of such measures.  The construction
produces nested point sets with controlled mesh norm, separation radius,
and optimal cardinality order. Combined with the discretization theorem
of Filbir and Mhaskar \cite{filbir2011marcinkiewicz}, it yields nested
deterministic \(L^p\) Marcinkiewicz--Zygmund inequalities for \(1\leq p<\infty\), in two regimes: a fixed-tolerance regime in which
\(n_j\asymp2^j\), and a fixed-degree refinement regime in which the Marcinkiewicz--Zygmund
tolerance tends to zero as the nested sets are refined.

\item
We use nested Marcinkiewicz--Zygmund measures to construct nearly tight semi-discrete
framelet systems from scattered sampling nodes.  For any prescribed
\(0<\eta<1\), the basic construction has frame bounds \(1-\eta\) and
\(1+\eta\).  For a fixed nested hierarchy, an oversampling shift \(s\)
improves the bounds to \(1-C_{\Mc}2^{-s}\) and
\(1+C_{\Mc}2^{-s}\).  Thus near-tightness can be improved without
relocating or discarding previously used nodes.

\item
We formulate the fully discrete filter-bank transform on the same
nested measures.  The Marcinkiewicz--Zygmund bounds control both the deviation of one-pass synthesis from the identity and the condition number of the canonical frame-operator equation.  We derive reconstruction and approximation estimates, develop the corresponding conjugate-gradient reconstruction, and record fast implementation and complexity estimates.

\end{enumerate}

\subsection{Organization}
The rest of the paper is organized as follows.  Section~2 introduces the
notation and the nested Marcinkiewicz--Zygmund definitions.  Section~3 gives the
deterministic construction and proves the corresponding \(L^p\)
inequalities.  Section~4 constructs framelets from nested Marcinkiewicz--Zygmund measures and
proves the semi-discrete frame bounds.  Section~5 develops the fully
discrete filter-bank transform, coefficient-space reconstruction, and
fast implementation.  Section~6 derives the corresponding polynomial
reconstruction and approximation estimates.  Section~7 presents
numerical experiments on the sphere and the torus, and Section~8
concludes the paper. 

\section{Notation and definitions}

For nonnegative quantities $A$ and $B$, we write $A\lesssim B$ if
$A\leq CB$ for some constant $C>0$, $A\gtrsim B$ if $B\lesssim A$, and $A\asymp B$ if both
hold.  Implicit constants are independent of the varying parameters but
may depend on the fixed manifold and other fixed parameters.

Let $(\Mc,g)$ be a compact, connected, $d$-dimensional smooth Riemannian
manifold without boundary, with Riemannian metric $g$. Let $\mu$ be the normalized Riemannian
volume measure on $\Mc$ with $\mu(\Mc)=1$.  We denote by $\dist$ the
geodesic distance on $\Mc$, that is,
\[
        \dist(x,y)
        :=
        \inf_{\omega(0)=x,\ \omega(1)=y}
        \int_0^1 |\dot\omega(t)|_g\,dt,
        \qquad x,y\in\Mc,
\]
where the infimum is taken over all piecewise smooth curves
$\omega:[0,1]\to\Mc$ joining $x$ and $y$.  For $R>0$, let $B(x,R):=\{y\in\Mc:\dist(x,y)\leq R\}$ be the closed geodesic ball.

We work with diffusion polynomials
associated with the Laplace--Beltrami operator $\Delta_\Mc$ on $\Mc$.  Let
$\{u_\ell\}_{\ell\geq0}$ be an orthonormal basis of $L^2(\mu)$ consisting
of eigenfunctions of $-\Delta_\Mc$ in the sense of
\[
        -\Delta_\Mc u_\ell=\lambda_\ell^2 u_\ell .
\]
The nondecreasing spectral parameters satisfy $0=\lambda_0\leq \lambda_1\leq\lambda_2\leq\cdots$ with $\lambda_\ell\to\infty$. For $f\in L^2(\mu)$, its generalized Fourier coefficients are
\[
        \widehat f_\ell:=\langle f,u_\ell\rangle_{L^2(\mu)}
        =
        \int_\Mc f(x)\overline{u_\ell(x)}\,d\mu(x),
        \qquad \ell\geq0.
\]
Then
\[
        f=\sum_{\ell=0}^{\infty}\widehat f_\ell u_\ell
        \quad\hbox{in }L^2(\mu),
        \qquad
        \|f\|_{L^2(\mu)}^2
        =
        \sum_{\ell=0}^{\infty}|\widehat f_\ell|^2 .
\]
For $n\geq1$, the diffusion polynomial space of bandwidth $n$ is
\[
        \Pd_n:=\operatorname{span}\{u_\ell:\lambda_\ell\leq n\}.
\]
On the sphere this is equivalent, up to the usual change of bandwidth
parameter, to the spherical polynomial spaces generated by spherical
harmonics of bounded degree.  Indeed, if \(\ell\) denotes the spherical
harmonic degree, then the Laplace--Beltrami spectral parameter is
\(\sqrt{\ell(\ell+d-1)}\).  Thus statements formulated with the harmonic degree
or with \(\lambda_\ell\) differ only by harmless constants in the
bandwidth.  In this compact manifold setting the standard Weyl-type
growth gives \(\dim\Pd_n\asymp n^d\).

We begin with the single-level notion.  A Marcinkiewicz--Zygmund measure for
$\Pd_n$ is a measure whose $L^p$ norm on $\Pd_n$ is equivalent to the
ambient manifold $L^p$ norm.

\begin{definition}[Marcinkiewicz--Zygmund measure]
Let $1\leq p<\infty$, $n\in\mathbb{N}$, and let $\nu$ be a positive
finite Borel measure on $\Mc$.  We say that $\nu$ is an
\emph{$L^p$ Marcinkiewicz--Zygmund measure} for $\Pd_n$ with constants $A,B>0$
if
\[
        A\|P\|_{L^p(\mu)}^p
        \leq
        \int_{\Mc} |P(x)|^p\,d\nu(x)
        \leq
        B\|P\|_{L^p(\mu)}^p,
        \qquad P\in\Pd_n .
\]
If $\nu$ is discrete, say $\nu=\sum_{x\in\Xc}w_x\delta_x$, 
then
the condition becomes
\[
        A\|P\|_{L^p(\mu)}^p
        \leq
        \sum_{x\in\Xc} w_x |P(x)|^p
        \leq
        B\|P\|_{L^p(\mu)}^p,
        \qquad P\in\Pd_n .
\]
\end{definition}

The nested object is obtained by requiring such inequalities at a
sequence of polynomial degrees, while also requiring the supporting node
sets to be nested.  The constants may either be uniform in the level or
improve along a refinement sequence, depending on the intended
application.

\begin{definition}[Nested Marcinkiewicz--Zygmund measures]
Let $\{n_j\}_{j\geq 0}$ be an increasing sequence of positive integers.
A sequence of discrete measures $\nu_j=\sum_{x\in\Xc_j} w_{j,x}\delta_x$ with $j\geq 0$
is called \emph{nested $L^p$ Marcinkiewicz--Zygmund measures} for
$\{\Pd_{n_j}\}_{j\geq 0}$ if
\[
        \Xc_j\subset\Xc_{j+1},\qquad j\geq 0,
\]
and there exist constants $A_j,B_j>0$ such that
\[
        A_j\|P\|_{L^p(\mu)}^p
        \leq
        \sum_{x\in\Xc_j} w_{j,x}|P(x)|^p
        \leq
        B_j\|P\|_{L^p(\mu)}^p,
        \qquad P\in\Pd_{n_j},\quad j\geq 0.
\]
\end{definition}

We shall use standard geometric parameters of finite point sets on
$\Mc$.  These parameters provide a convenient deterministic route to Marcinkiewicz--Zygmund
inequalities. For a finite set $\Xc\subset\Mc$, we define its
\emph{mesh norm} and \emph{separation distance} by
\[
h_\Xc:=\sup_{y\in\Mc}\min_{x\in\Xc}\dist(y,x)\qquad\text{and}\qquad\sepa(\Xc):=\min_{x,y\in\Xc,\ x\neq y}\dist(x,y),
\]
respectively. The mesh norm measures how well $\Xc$ covers the
manifold: every point of $\Mc$ lies within distance $h_\Xc$ of some node
of $\Xc$.  The separation distance measures how far distinct nodes are
from each other, and rules out excessive clustering when it is bounded
from below.

For multiscale applications the most relevant spatial scale is dyadic.
We therefore emphasize the dyadic geometry of the nodes. These node sets are nested, cover $\Mc$ at dyadic resolution, and
remain uniformly separated at the same scale.

\begin{definition}[Dyadic nested sampling family]
A sequence of finite point sets $\{\Xc_j\}_{j\geq0}$ on $\Mc$ is called
a \emph{dyadic nested sampling family} if
\[
        \Xc_j\subset\Xc_{j+1},\qquad j\geq0,
\]
and there exist constants $c,C>0$, independent of $j$, such that
\[
        h_{\Xc_j}\leq C2^{-j},
        \qquad
        \sepa(\Xc_j)\geq c2^{-j},
        \qquad
        |\Xc_j|\asymp 2^{jd}.
\]
\end{definition}

Since $\Mc$ is compact, its injectivity radius is positive and its
curvature is bounded.  Hence there exist constants $R_\Mc>0$ and
$c_1,c_2>0$, depending only on $\Mc$, such that
\begin{equation}
\label{eq:local-volume-estimate}
        c_1 R^d\leq \mu(B(x,R))\leq c_2 R^d,
        \qquad x\in\Mc,\quad 0<R\leq R_\Mc .
\end{equation}
This standard local volume estimate on compact Riemannian manifolds can be found in, e.g., \cite[Chapter~III, Sections~3--4]{chavel2006riemannian}.

\section{Deterministic dyadic nested sampling families}\label{sec:deterministic_dyadic_nested_sampling_families}

In this section, we construct dyadic nested point sets with controlled
geometry and then establish nested
Marcinkiewicz--Zygmund measures.  

\subsection{Dyadic nested quasi-uniform sets}

Let
$0<\gamma<R_\Mc$, where $R_\Mc$ is the constant in the estimate \eqref{eq:local-volume-estimate}, and
set $n_j:=2^j$ and $R_j:=\gamma n_j^{-1}$ for $j\geq0$.

\begin{lemma}[Nested quasi-uniform sets]
\label{lem:nested-quasi-uniform}
There exists a sequence of finite subsets
$\{\Xc_j\}_{j\geq0}$ of $\Mc$ such that, for any $j\geq0$,
\[
        \Xc_j\subset\Xc_{j+1},
\]
and
\[
        h_{\Xc_j}\leq R_j,\qquad
        \sepa(\Xc_j)\geq R_j,\qquad
        |\Xc_j|\asymp R_j^{-d}\asymp \gamma^{-d}2^{jd}.
\]
The constants in $\asymp$ depend only on $\Mc$.
\end{lemma}

\begin{proof}
We construct the sets inductively.  A point set
$\boldsymbol{Y}\subset\Mc$ is called maximal $R$-separated if it is $R$-separated and no point of $\Mc\setminus\boldsymbol{Y}$ can be added
while preserving $R$-separation.  Choose $\Xc_0$ to be a maximal
$R_0$-separated subset of $\Mc$.  Then $\sepa(\Xc_0)\geq R_0$.  Moreover,
maximality implies the covering property: if some $y\in\Mc$ satisfied
$\dist(y,x)>R_0$ for all $x\in\Xc_0$, then
$\Xc_0\cup\{y\}$ would still be $R_0$-separated, a contradiction.  Hence
$h_{\Xc_0}\leq R_0$. Now assume that $\Xc_j$ has been constructed.  Since
$R_{j+1}=R_j/2$ and $\sepa(\Xc_j)\geq R_j$, the old set $\Xc_j$ is
$R_{j+1}$-separated.  Enlarge $\Xc_j$, without removing any old point,
to a maximal $R_{j+1}$-separated set and call the result
$\Xc_{j+1}$.  Compactness ensures that this enlargement is finite.  By
construction,
\[
        \Xc_j\subset\Xc_{j+1},
        \qquad
        \sepa(\Xc_{j+1})\geq R_{j+1}.
\]
The same maximality argument gives $h_{\Xc_{j+1}}\leq R_{j+1}$.  This
proves the nestedness, mesh norm, and separation estimates.

It remains only to estimate the cardinality.  Since
$R_j=\gamma 2^{-j}\leq \gamma<R_\Mc$, the local volume estimate
\eqref{eq:local-volume-estimate} gives
$\mu(B(x,R))\asymp R^d$ for $0<R\leq R_\Mc$, with constants
depending only on $\Mc$.  Since $\Xc_j$ is $R_j$-separated, the balls
$B(x,R_j/3)$, $x\in\Xc_j$, are pairwise disjoint.  Therefore
\[
        |\Xc_j|R_j^d\lesssim
        \sum_{x\in\Xc_j}\mu(B(x,R_j/3))
        \leq 1,
\]
and so $|\Xc_j|\lesssim R_j^{-d}$.  Conversely, since
$h_{\Xc_j}\leq R_j$, the balls $B(x,R_j)$, $x\in\Xc_j$, cover $\Mc$.
Thus
\[
        1\leq \sum_{x\in\Xc_j}\mu(B(x,R_j))
        \lesssim |\Xc_j|R_j^d,
\]
which gives $|\Xc_j|\gtrsim R_j^{-d}$.  Combining the two bounds yields
$|\Xc_j|\asymp R_j^{-d}\asymp \gamma^{-d}2^{jd}$.
\end{proof}

\subsection{Dyadic partitions and weights}

Geodesic Voronoi diagrams on Riemannian manifolds are standard; see, e.g., \cite{leibon2000delaunay}.  For the weighted discretization
below, we use a measurable version obtained by assigning points on
Voronoi boundaries according to a fixed ordering of the nodes.

\begin{definition}[Partition weights]
\label{def:partition-weights}
Let $\Xc_j\subset\Mc$ be a finite point set.  A measurable partition
\(\{Y_{j,x}:x\in\Xc_j\}\) of \(\Mc\) is called a
\emph{nearest-neighbor partition associated with \(\Xc_j\)} if
\[
        \Mc=\bigcup_{x\in\Xc_j}Y_{j,x},
        \qquad x\in Y_{j,x},
\]
and
\[
        Y_{j,x}\subset
        \left\{y\in\Mc:\dist(y,x)=
        \min_{z\in\Xc_j}\dist(y,z)\right\},
        \qquad
        Y_{j,x}\subset B(x,h_{\Xc_j}),
\]
where points tied between several nearest nodes are assigned to the
first such node in a fixed ordering of \(\Xc_j\).  This produces a Borel
partition because the geodesic distance is continuous and \(\Xc_j\) is
finite.  The corresponding
\emph{partition weights} are
\[
        w_{j,x}:=\mu(Y_{j,x}),\qquad x\in\Xc_j.
\]
\end{definition}

The next lemma shows that, for the dyadic nested sampling family
constructed above, these geometrically defined weights have the expected
volume scale at every level.

\begin{lemma}[Dyadic weights]
\label{lem:dyadic-weights}
Let $\{\Xc_j\}_{j\geq0}$ be the dyadic nested sampling family
constructed in Lemma~\ref{lem:nested-quasi-uniform}, and equip each
$\Xc_j$ with the partition weights from
Definition~\ref{def:partition-weights}.  Then
\[
        w_{j,x}\asymp R_j^d\asymp \gamma^d2^{-jd},
        \qquad x\in\Xc_j,\quad j\geq0.
\]
The constants depend only on $\Mc$.
\end{lemma}

\begin{proof}
Fix $j\geq0$ and $x\in\Xc_j$.  The upper bound uses only the mesh norm.
By construction of the nearest-neighbor partition, $Y_{j,x}\subset B(x,h_{\Xc_j})$.
Lemma~\ref{lem:nested-quasi-uniform} gives $h_{\Xc_j}\leq R_j$, and hence $Y_{j,x}\subset B(x,R_j)$.
The local volume estimate \eqref{eq:local-volume-estimate} therefore gives
\[
        w_{j,x}=\mu(Y_{j,x})
        \leq \mu(B(x,R_j))
        \lesssim R_j^d.
\]
For the lower bound we use the separation.  Let $y\in\Mc$ satisfy
$\dist(y,x)<R_j/2$, and let $z\in\Xc_j$ with $z\neq x$.  Since
$\sepa(\Xc_j)\geq R_j$, we have $\dist(x,z)\geq R_j$.  By the triangle
inequality,
\[
        \dist(y,z)\geq \dist(x,z)-\dist(y,x)
        > R_j-R_j/2=R_j/2.
\]
On the other hand, $\dist(y,x)<R_j/2$.  Thus $x$ is the unique nearest
node to every point in the open ball
$B^\circ(x,R_j/2):=\{y:\dist(y,x)<R_j/2\}$.  Hence
\[
        w_{j,x}\geq \mu(B^\circ(x,R_j/2))
        \geq \mu(B(x,R_j/3))
        \gtrsim R_j^d,
\]
where the last inequality follows from
\eqref{eq:local-volume-estimate}.
Combining the two estimates gives $w_{j,x}\asymp R_j^d$.  Since
$R_j=\gamma2^{-j}$, this is the same as
$w_{j,x}\asymp \gamma^d2^{-jd}$.
\end{proof}

\subsection{Deterministic Marcinkiewicz--Zygmund inequalities}

We now recall the form of the Filbir--Mhaskar discretization theorem \cite[Theorem 5.1(b)]{filbir2011marcinkiewicz} on
compact manifolds.  The statement is specialized to the
nearest-neighbor partitions used in this paper.

\begin{proposition}[Filbir--Mhaskar \cite{filbir2011marcinkiewicz}]
\label{prop:filbir-mhaskar}
Let $1\leq p<\infty$, $0<\eta<1$, $\kappa\geq1$, and $n\geq1$.  There exists a
constant $c=c(\Mc,\kappa,p)>0$ with the following property.  Let
 $\Xc\subset\Mc$ be a finite set satisfying the
quasi-uniformity condition $h_\Xc\leq \kappa\,\sepa(\Xc)$, 
and
$\{Y_x\}_{x\in\Xc}$ a measurable partition of $\Mc$ such that
\[
        x\in Y_x,
        \qquad
        Y_x\subset B(x,h_\Xc),
        \qquad x\in\Xc.
\]
If $n h_\Xc\leq c\,\eta/p$,
then, for every real- or complex-valued $P\in\Pd_n$,
\[
        \left|
        \int_{\Mc}|P(y)|^p\,d\mu(y)
        -
        \sum_{x\in\Xc}\mu(Y_x)|P(x)|^p
        \right|
        \leq
        \eta\int_{\Mc}|P(y)|^p\,d\mu(y).
\]
\end{proposition}

The condition in Proposition~\ref{prop:filbir-mhaskar} involves three
quantities: the polynomial degree $n$, the mesh norm $h_\Xc$, and the Marcinkiewicz--Zygmund
tolerance $\eta$.  In a nested construction, the mesh norm changes from
level to level.  This gives two natural ways to read the result.  The
first one fixes a tolerance $\eta$ and chooses the degree at level $j$
so that $n_jh_{\Xc_j}$ stays uniformly small.  The second one fixes a
degree $n$ and lets the tolerance decrease as the nested sets become
finer.

\begin{theorem}[Fixed-tolerance dyadic nested Marcinkiewicz--Zygmund measures]
\label{thm:fixed-tolerance-mz}
Let $1\leq p<\infty$ and $0<\eta<1$.  There exist
$\gamma_{\eta,p,\Mc}>0$, dyadic
nested point sets $\Xc_j\subset\Xc_{j+1}$, $n_j=2^j$, and positive
weights $\{w_{j,x}\}_{x\in\Xc_j}$ such that
\[
        |\Xc_j|\asymp \gamma_{\eta,p,\Mc}^{-d}n_j^d,\qquad
        w_{j,x}\asymp \gamma_{\eta,p,\Mc}^d n_j^{-d},
        \qquad x\in\Xc_j,\quad j\geq0,
\]
and
\[
        (1-\eta)\|P\|_{L^p(\mu)}^p
        \leq
        \sum_{x\in\Xc_j} w_{j,x}|P(x)|^p
        \leq
        (1+\eta)\|P\|_{L^p(\mu)}^p
\]
for every $P\in\Pd_{n_j}$ and every $j\geq0$.  The parameter
$\gamma_{\eta,p,\Mc}$ may be chosen so that
\begin{equation}\label{equ:gammabound}
        0<\gamma_{\eta,p,\Mc}<\min\{R_\Mc,c\,\eta/p\},
\end{equation}
where $c$ is the constant in Proposition~\ref{prop:filbir-mhaskar} for
$\kappa=1$. The implicit constants in the cardinality and
weight estimates depend only on $\Mc$; the dependence on $\eta$ and $p$
is carried by the explicit factor $\gamma_{\eta,p,\Mc}$.
\end{theorem}

\begin{proof}
Let $c$ be the constant in Proposition~\ref{prop:filbir-mhaskar} with
$\kappa=1$.  Choose $\gamma_{\eta,p,\Mc}>0$ so small that \eqref{equ:gammabound} holds.
Construct $\Xc_j$ with $R_j=\gamma_{\eta,p,\Mc} n_j^{-1}$ and define
$w_{j,x}=\mu(Y_{j,x})$ as above. For each $j$, Lemma~\ref{lem:nested-quasi-uniform} gives
$h_{\Xc_j}\leq R_j$ and $\sepa(\Xc_j)\geq R_j$.  Hence
$h_{\Xc_j}\leq\sepa(\Xc_j)$, so the quasi-uniformity condition of
Proposition~\ref{prop:filbir-mhaskar} holds with $\kappa=1$.  Moreover,
$h_{\Xc_j}\leq R_j=\gamma_{\eta,p,\Mc} n_j^{-1}$, and therefore
$n_jh_{\Xc_j}\leq\gamma_{\eta,p,\Mc}\leq c\eta/p$.
Proposition~\ref{prop:filbir-mhaskar} therefore gives, for every
$P\in\Pd_{n_j}$,
\[
        \left|
        \int_{\Mc}|P(y)|^p\,d\mu(y)
        -
        \sum_{x\in\Xc_j}\mu(Y_{j,x})|P(x)|^p
        \right|
        \leq
        \eta\int_{\Mc}|P(y)|^p\,d\mu(y).
\]
This is exactly the stated Marcinkiewicz--Zygmund inequality.  The
cardinality and weight estimates follow from
Lemmas~\ref{lem:nested-quasi-uniform} and \ref{lem:dyadic-weights},
since $R_j=\gamma_{\eta,p,\Mc} n_j^{-1}$.
\end{proof}

\begin{theorem}[Fixed-degree dyadic nested Marcinkiewicz--Zygmund measures]
\label{thm:fixed-degree-refinement}
Let $1\leq p<\infty$ and let $n\geq1$ be fixed.  There exist dyadic
nested point sets $\Xc_j\subset\Xc_{j+1}$ and positive weights
$\{w_{j,x}\}_{x\in\Xc_j}$, constructed with a fixed parameter
$\gamma=\gamma(\Mc)$, such that
\[
        |\Xc_j|\asymp \gamma^{-d}2^{jd},
        \qquad
        w_{j,x}\asymp \gamma^d2^{-jd},
        \qquad x\in\Xc_j,\quad j\geq0,
\]
and, for all sufficiently large $j$,
\[
        (1-\eta_j)\|P\|_{L^p(\mu)}^p
        \leq
        \sum_{x\in\Xc_j} w_{j,x}|P(x)|^p
        \leq
        (1+\eta_j)\|P\|_{L^p(\mu)}^p,
        \qquad P\in\Pd_n,
\]
where $\eta_j\leq C_{\Mc,p}\,n\,2^{-j}$. In particular, the Marcinkiewicz--Zygmund constants tend to $1$ as $j\to\infty$.
\end{theorem}

\begin{proof}
Let $c$ be the constant in Proposition~\ref{prop:filbir-mhaskar} with
$\kappa=1$.  Fix once and for all a number $0<\gamma<R_\Mc$, for
instance $\gamma=R_\Mc/2$.  Construct the dyadic nested sets with
$R_j=\gamma2^{-j}$, and define
$w_{j,x}=\mu(Y_{j,x})$ as in Lemma~\ref{lem:dyadic-weights}.
For each $j$, Lemma~\ref{lem:nested-quasi-uniform} gives
$h_{\Xc_j}\leq R_j=\gamma2^{-j}$ and
$h_{\Xc_j}\leq\sepa(\Xc_j)$.  Define
\[
        \eta_j:=\frac{p}{c}\,n\,h_{\Xc_j}.
\]
Since $h_{\Xc_j}\leq\gamma2^{-j}$, we have
\[
        \eta_j\leq \frac{p\gamma}{c}\,n\,2^{-j}.
\]
Because $\gamma$ depends only on $\Mc$, the constant $p\gamma/c$ may be
absorbed into $C_{\Mc,p}$.
This also shows that $\eta_j\to0$ as $j\to\infty$, and hence
$\eta_j<1$ for all sufficiently large $j$.  For such $j$, $n h_{\Xc_j}=c\,\eta_j/p$,
so Proposition~\ref{prop:filbir-mhaskar} applies to $\Pd_n$ with
tolerance $\eta_j$ and gives the desired Marcinkiewicz--Zygmund inequality.  The cardinality
and weight estimates follow from Lemmas~\ref{lem:nested-quasi-uniform}
and \ref{lem:dyadic-weights}, since $R_j=\gamma2^{-j}$.
\end{proof}

\section{Nearly tight framelet systems from nested Marcinkiewicz--Zygmund measures}\label{sec:framelets_from_nested_marcinkiewicz_zygmund_arrays}

We now turn the nested sampling inequalities into multiscale systems.
The continuous framelet construction provides spectral building blocks
at dyadic scales, with a center variable $y$ ranging over the whole
manifold $\Mc$. To obtain a usable semi-discrete system, this center variable
must be sampled at each scale.

\subsection{Continuous framelets on compact manifolds}

We first recall the continuous framelet construction on a compact
manifold, following the spectral approach in \cite{wang2020tight}.  Let
$\Psi:=\{\alpha;\beta^1,\ldots,\beta^r\}$
be a finite collection of \emph{framelet generators} in $L_1(\R)$ whose
Fourier transforms are compactly supported on the nonnegative frequency
axis.  Throughout this section we use the \(2\pi\)-Fourier convention
$\widehat h(\xi)
        :=
        \int_\R h(t)e^{-2\pi i t\xi}\,dt$.
We fix an arbitrary starting scale $J\in\mathbb{N}_0$, and define, for $j\geq J$ and $y\in\Mc$, the low- and
high-pass continuous framelet kernels by
\begin{equation}
\label{eq:low-pass-kernel}
        \varphi_{j,y}(x)
        :=
        \sum_{\ell=0}^{\infty}
        \widehat\alpha\!\left(\frac{\lambda_\ell}{2^j}\right)
        \overline{u_\ell(y)}\,u_\ell(x),
        \qquad x\in\Mc,
\end{equation}
and
\begin{equation}
\label{eq:high-pass-kernel}
        \psi^k_{j,y}(x)
        :=
        \sum_{\ell=0}^{\infty}
        \widehat{\beta^k}\!\left(\frac{\lambda_\ell}{2^j}\right)
        \overline{u_\ell(y)}\,u_\ell(x),
        \qquad x\in\Mc,\quad k=1,\ldots,r,
\end{equation}
respectively. The variable $y\in \Mc$ plays the role of the translation parameter. 
\begin{definition}[Continuous framelet system]
\label{def:continuous-framelet-system}
Let $J_0\in\mathbb{N}_0$ be an integer. For any $J\geq J_0$, the family
\[
        \mathrm{CFS}_J(\Psi)
        :=
        \{\varphi_{J,y}:y\in\Mc\}
        \cup
        \{\psi^k_{j,y}:y\in\Mc,\ j\geq J,\ k=1,\ldots,r\}
\]
is called the \emph{continuous framelet system} generated by $\Psi$ from
scale $J$.
\end{definition}

The system $\mathrm{CFS}_J(\Psi)$ is called a continuous \emph{tight}
frame for $L^2(\Mc,\mu)$ if its elements belong to $L^2(\Mc,\mu)$ and
if, for every $f\in L^2(\Mc,\mu)$,
\begin{equation}
\label{eq:continuous-tight-frame-identity}
        \|f\|_{L^2(\mu)}^2
        =
        \int_{\Mc}
        |\langle f,\varphi_{J,y}\rangle|^2\,d\mu(y)
        +
        \sum_{j=J}^{\infty}\sum_{k=1}^r
        \int_{\Mc}
        |\langle f,\psi^k_{j,y}\rangle|^2\,d\mu(y).
\end{equation}
The tightness condition \eqref{eq:continuous-tight-frame-identity} can be verified in a spectral manner.  Assume that, for every
$\ell\geq0$,
\begin{equation}
\label{eq:spectral-lowpass-limit}
        \lim_{j\to\infty}
        \left|
        \widehat\alpha\!\left(\frac{\lambda_\ell}{2^j}\right)
        \right|
        =1,
\end{equation}
and that, for every $j\geq J$ and every $\ell\geq0$,
\begin{equation}
\label{eq:spectral-refinement}
        \left|
        \widehat\alpha\!\left(\frac{\lambda_\ell}{2^{j+1}}\right)
        \right|^2
        =
        \left|
        \widehat\alpha\!\left(\frac{\lambda_\ell}{2^j}\right)
        \right|^2
        +
        \sum_{k=1}^r
        \left|
        \widehat{\beta^k}\!\left(\frac{\lambda_\ell}{2^j}\right)
        \right|^2 .
\end{equation}
Then the identities \eqref{eq:spectral-refinement} telescope to
\begin{equation}
\label{eq:spectral-partition}
        \left|
        \widehat\alpha\!\left(\frac{\lambda_\ell}{2^J}\right)
        \right|^2
        +
        \sum_{j=J}^{\infty}\sum_{k=1}^r
        \left|
        \widehat{\beta^k}\!\left(\frac{\lambda_\ell}{2^j}\right)
        \right|^2
        =
        1,\qquad \ell\geq0.
\end{equation}
Consequently, $\mathrm{CFS}_J(\Psi)$ is a continuous tight frame.  To see
this, first let $f$ be a finite spectral expansion,
$f=\sum_{\ell\geq0}\widehat f_\ell u_\ell$.  Then
\begin{equation}
\label{eq:framelet-coefficient-expansions}
 \langle f,\varphi_{j,y}\rangle
        =
        \sum_{\ell\geq0}
        \widehat f_\ell\,
        \overline{
        \widehat\alpha\!\left(\frac{\lambda_\ell}{2^j}\right)}
        u_\ell(y),\qquad
\langle f,\psi^k_{j,y}\rangle
        =
        \sum_{\ell\geq0}
        \widehat f_\ell\,
        \overline{
        \widehat{\beta^k}\!\left(\frac{\lambda_\ell}{2^j}\right)}
        u_\ell(y).
\end{equation}
Parseval's identity in the $y$ variable gives
\[
        \int_\Mc |\langle f,\varphi_{J,y}\rangle|^2\,d\mu(y)
        =
        \sum_{\ell\geq0}
        |\widehat f_\ell|^2
        \left|
        \widehat\alpha\!\left(\frac{\lambda_\ell}{2^J}\right)
        \right|^2
\]
and the analogous identity for each high-pass term.  Summing over the
high-pass levels and using \eqref{eq:spectral-partition}, the continuous
framelet energy is exactly $\sum_{\ell\geq0}|\widehat f_\ell|^2$.
For a general $f\in L^2(\Mc,\mu)$, the same identity follows by
approximating $f$ by finite spectral sums and using Tonelli's theorem
for the nonnegative energy terms.  Further details and equivalent
conditions can be found in \cite[Theorem 2.1]{wang2020tight}.

\subsection{Nested Marcinkiewicz--Zygmund sampling and semi-discrete framelets}

The continuous system \(\mathrm{CFS}_J(\Psi)\) uses discrete scales and
continuous centers.  Semi-discretization further replaces the center continuum
at each scale by finitely many weighted nodes while preserving
\(L^2(\Mc,\mu)\) frame bounds.  Exact-quadrature constructions in
\cite{wang2020tight,xiao2023spherical} preserve tightness but require
structured rules, such as spherical \(t\)-designs.  Here we instead use
the nested measures from
Section~\ref{sec:deterministic_dyadic_nested_sampling_families}, whose Marcinkiewicz--Zygmund stability also allows geometrically controlled scattered and nested
nodes.

For a fixed function \(f\), the quantities to be discretized are the
center-variable coefficient functions
\begin{equation}\label{equ:coefficientfunctions}
        P_J^\varphi(y):=\langle f,\varphi_{J,y}\rangle,
        \qquad
        P_j^{\psi^k}(y):=\langle f,\psi^k_{j,y}\rangle .
\end{equation}
To use the Marcinkiewicz--Zygmund measures,
these functions should lie in finite-dimensional diffusion polynomial
spaces.  Similar to the exact quadrature-based discretization
\cite[Corollary 2.6]{wang2020tight}, we ensure this by imposing a
compact spectral support condition on the generators.  Namely, we assume
that \(\Psi\) satisfies \eqref{eq:spectral-lowpass-limit} and
\eqref{eq:spectral-refinement}, and 
\begin{equation}
\label{eq:filter-unit-support}
        \left(\supp \widehat\alpha\cap[0,\infty)\right)
        \cup
        \bigcup_{k=1}^r
        \left(\supp \widehat{\beta^k}\cap[0,\infty)\right)
        \subset [0,1].
\end{equation}
If a multiplier at scale $j$ is supported in $[0,1]$ after the
normalization $\lambda_\ell\mapsto \lambda_\ell/2^j$, then only
eigenvalues $\lambda_\ell\leq 2^j$ can occur.  The next lemma records
the resulting bandwidth of the coefficient functions \eqref{equ:coefficientfunctions}.

\begin{lemma}[Bandwidth of coefficient functions]
\label{lem:coefficient-bandwidth}
Assume that \(\Psi\) satisfies \eqref{eq:filter-unit-support}. For every $f\in L^2(\Mc,\mu)$,
\begin{equation}
\label{eq:coefficient-bandwidth}
        P_J^\varphi\in \Pd_{2^J},
        \qquad
        P_j^{\psi^k}\in \Pd_{2^j},
        \qquad j\geq J,\quad k=1,\ldots,r .
\end{equation}
\end{lemma}

\begin{proof}
By the coefficient expansion \eqref{eq:framelet-coefficient-expansions},
\[
        P_J^\varphi(y)
        =
        \sum_{\ell\geq0}
        \widehat f_\ell\,
        \overline{
        \widehat\alpha\!\left(\frac{\lambda_\ell}{2^J}\right)}
        u_\ell(y).
\]
The support condition \eqref{eq:filter-unit-support} implies that
the coefficient multiplying $u_\ell$ is zero unless
$\lambda_\ell/2^J\leq 1$, or equivalently
$\lambda_\ell\leq 2^J$.  Therefore
$P_J^\varphi\in\Pd_{2^J}$.  The proof for
$P_j^{\psi^k}$ is similar.
\end{proof}

\begin{remark}[On products of diffusion polynomials]
\label{rem:products-diffusion-polynomials}
Since the product of two diffusion polynomials may not necessarily be a diffusion polynomial,
exact-quadrature discretizations often require a product assumption,
such as \(|P|^2\in\Pd_{Cn}\) for \(P\in\Pd_n\); see, e.g.,
\cite{wang2020tight}.  Filbir and Mhaskar proved the existence of a manifold-dependent
constant \(C_{\mathrm{prod}}\geq2\) such that
\[
        P\overline Q\in\Pd_{C_{\mathrm{prod}}2^j}
        \qquad\text{for all }P,Q\in\Pd_{2^j};
\]
see 
\cite[Theorem A.1]{filbir2011marcinkiewicz}.
Consequently, exactness through degree \(C_{\mathrm{prod}}2^j\) makes the level-\(j\)
weighted Fourier map an isometry.  On the sphere and the flat torus one
may take \(C_{\mathrm{prod}}=2\), so exactness through degree \(2^{j+1}\) is enough.
On a general compact manifold, however, \(C_{\mathrm{prod}}\) need not equal \(2\),
and the quadrature degree must be enlarged accordingly.  Indeed, the
product inclusion and the corresponding exactness rule give
\(\sum_{x\in\Xc_j}w_{j,x}P(x)\overline{Q(x)}
=\langle P,Q\rangle_{L^2(\mu)}\).  Thus, exactness alone is not the
relevant condition, and it must be paired with a product-bandwidth
inclusion of this form.
The Marcinkiewicz--Zygmund approach in the current paper, however, needs no such hypothesis:
Proposition~\ref{prop:filbir-mhaskar} directly provides the required
two-sided estimate for \(|P|^2\).
\end{remark}

Since the Marcinkiewicz--Zygmund inequalities are weighted, we absorb
\(w_x^{1/2}\) into each sampled atom.  The resulting unweighted
\(\ell^2\)-sum of framelet coefficients then coincides with the
corresponding weighted Marcinkiewicz--Zygmund sum.  We first define this
semi-discrete system and then establish its frame property.

\begin{definition}[Semi-discrete framelets from nested Marcinkiewicz--Zygmund measures]
\label{def:semi-discrete-framelet-system}
Let $J_0\in\mathbb{N}_0$ and $J\geq J_0$.  For each $j\geq0$, let $\Xc_j$ be a finite
subset of $\Mc$ and let $\{w_{j,x}\}_{x\in\Xc_j}$ be positive weights.
Assume that the sampling family is nested, that is,
$\Xc_j\subset\Xc_{j+1}$ for all $j\geq0$.  The family
\begin{equation}\label{eq:nested-framelet-system}
        \mathrm{FS}_J(\Psi;\Xc,w)
        =
        \{w_{J,x}^{1/2}\varphi_{J,x}:x\in\Xc_J\}
        \cup
        \{w_{j+1,x}^{1/2}\psi_{j,x}^k:
        x\in\Xc_{j+1},\ j\geq J,\ k=1,\ldots,r\}
\end{equation}
is called the \emph{semi-discrete framelet system from nested Marcinkiewicz--Zygmund measures},
associated with the continuous framelet system $\mathrm{CFS}_J(\Psi)$.
\end{definition}

Recall that a semi-discrete framelet system $\mathrm{FS}_J$ is a \emph{frame} for
$L^2(\Mc,\mu)$ with frame bounds $0<A\leq B<\infty$ if its framelet
coefficients satisfy
\begin{equation*}
        A\|f\|_{L^2(\mu)}^2
        \leq
        \sum_{g\in\mathrm{FS}_J}
        |\langle f,g\rangle|^2
        \leq
        B\|f\|_{L^2(\mu)}^2
\end{equation*}
for every $f\in L^2(\Mc,\mu)$.  If $A=B=1$, the system is \emph{tight}. With the normalization inherited from the continuous tight frame, we
call the system \(\varepsilon\)-\emph{nearly tight} if
\[
(1-\varepsilon)\|f\|_{L^2(\mu)}^2
\leq
\sum_{g\in\mathrm{FS}_J}|\langle f,g\rangle|^2
\leq
(1+\varepsilon)\|f\|_{L^2(\mu)}^2,
\qquad 0\leq\varepsilon<1.
\]
Equivalently, its frame operator \(S\) satisfies
\(\|S-I\|\leq\varepsilon\).

\begin{theorem}[Nearly tight framelet systems from nested Marcinkiewicz--Zygmund measures]
\label{thm:nested-framelets}
Let $J_0\in\mathbb{N}_0$ and $J\geq J_0$.  Assume that $\Psi$ satisfies
\eqref{eq:spectral-lowpass-limit}, \eqref{eq:spectral-refinement}, and
\eqref{eq:filter-unit-support}.  Let \(0<\eta<1\), and form the
nested semi-discrete framelet system \eqref{eq:nested-framelet-system}
from the fixed-tolerance nested Marcinkiewicz--Zygmund measures in
Theorem~\ref{thm:fixed-tolerance-mz}, applied with \(p=2\) and
constructed with the parameter \(\gamma_{\eta,2,\Mc}\).  Then \(\mathrm{FS}_J(\Psi;\Xc,w)\) is an
\(\eta\)-nearly tight frame for \(L^2(\Mc,\mu)\), with frame bounds
\(1-\eta\) and \(1+\eta\).
\end{theorem}

\begin{proof}
For fixed $f\in L^2(\Mc,\mu)$, by Lemma~\ref{lem:coefficient-bandwidth},
$P_J^\varphi\in\Pd_{2^J}$ and
$P_j^{\psi^k}\in\Pd_{2^j}$.  Since the level-$q$ rule in
Theorem~\ref{thm:fixed-tolerance-mz} is a Marcinkiewicz--Zygmund rule for $\Pd_{2^q}$,
we apply it at level $q=J$ to the low-pass term.  For a high-pass term
at scale $j$, we use the level-$(j+1)$ rule; this is valid because
$\Pd_{2^j}\subset\Pd_{2^{j+1}}$.  The low-pass estimate is
\[
        (1-\eta)\int_\Mc |P_J^\varphi(y)|^2\,d\mu(y)
        \leq
        \sum_{x\in\Xc_J}
        w_{J,x}|P_J^\varphi(x)|^2
        \leq
        (1+\eta)\int_\Mc |P_J^\varphi(y)|^2\,d\mu(y),
\]
which is the same as
\begin{equation}\label{equ:lowestimate}
        (1-\eta)\|P_J^\varphi\|_{L^2(\mu)}^2
        \leq
        \sum_{x\in\Xc_J}
        |\langle f,w_{J,x}^{1/2}\varphi_{J,x}\rangle|^2
        \leq
        (1+\eta)\|P_J^\varphi\|_{L^2(\mu)}^2.
\end{equation}
Similarly, for every $j\geq J$ and $k=1,\ldots,r$,
\begin{equation}\label{equ:highestimate}
        (1-\eta)\int_\Mc |P_j^{\psi^k}(y)|^2\,d\mu(y)
        \leq
        \sum_{x\in\Xc_{j+1}}
        |\langle f,w_{j+1,x}^{1/2}\psi_{j,x}^k\rangle|^2
        \leq
        (1+\eta)\int_\Mc |P_j^{\psi^k}(y)|^2\,d\mu(y).
\end{equation}
Summing the low-pass estimate \eqref{equ:lowestimate} and the high-pass estimates \eqref{equ:highestimate} over
$J\leq j\leq L$ gives the same two-sided inequality for the corresponding
finite partial coefficient energy.  Since all terms are nonnegative, we
may let $L\to\infty$ and use monotone convergence on the discrete sums.
On the continuous side, the corresponding partial energies converge to
$\|f\|_{L^2(\mu)}^2$ by the tight identity
\eqref{eq:continuous-tight-frame-identity}.  This proves the frame
bounds $1-\eta$ and $1+\eta$ for $\mathrm{FS}_J(\Psi;\Xc,w)$.
\end{proof}

\begin{remark}[Sampling cost of near-tightness]
The tolerance \(\eta\) in Theorem~\ref{thm:nested-framelets} is the same
at every level, but it can be chosen when constructing the measures.
Choosing a smaller \(\eta\) gives tighter frame bounds \(1\pm\eta\).
The price is a denser sampling family.  Indeed, by
Theorem~\ref{thm:fixed-tolerance-mz} with \(p=2\) and \(n_j=2^j\),
the construction gives
\[
        |\Xc_j|\asymp \gamma_{\eta,2,\Mc}^{-d}2^{jd},
        \qquad
        w_{j,x}\asymp \gamma_{\eta,2,\Mc}^{d}2^{-jd},
        \qquad x\in\Xc_j,\quad j\geq0 .
\]
Here the construction parameter is chosen subject to $0<\gamma_{\eta,2,\Mc}<\min\{R_\Mc,c\eta/2\}$, where \(c\) is the constant in Proposition~\ref{prop:filbir-mhaskar}
with \(p=2\) and \(\kappa=1\).
Thus improving the frame bounds by decreasing \(\eta\) generally
amplifies the number of sampling centers through the factor
\(\gamma_{\eta,2,\Mc}^{-d}\), which is of order \(\eta^{-d}\) when
\(\gamma_{\eta,2,\Mc}\) is chosen proportional to \(\eta\).
A high-pass
term at scale \(j\) uses \(\Xc_{j+1}\), and hence has $|\Xc_{j+1}|\asymp
        \gamma_{\eta,2,\Mc}^{-d}2^{(j+1)d}$ sampling centers.
\end{remark}

There is another way to improve the frame bounds once a nested family
has been constructed.  Instead of rebuilding the whole family with a
smaller value of \(\gamma_{\eta,2,\Mc}\), one may keep the same nested
sets and add \(s\) refinement steps beyond the basic sampling levels.
Thus the low-pass term at scale \(J\) is sampled on \(\Xc_{J+s}\),
whereas a high-pass coefficient function at scale \(j\), whose basic
center set is \(\Xc_{j+1}\), is sampled on \(\Xc_{j+1+s}\). This is
the sense in which the next construction is oversampled: all sampling
levels in the basic semi-discrete system are shifted upward by \(s\).

\begin{definition}[Oversampled semi-discrete framelets from nested Marcinkiewicz--Zygmund measures]
\label{def:oversampled-framelet-system}
Let $J_0\in\mathbb{N}_0$ and $J\geq J_0$.  Assume that $\Psi$ satisfies
\eqref{eq:spectral-lowpass-limit}, \eqref{eq:spectral-refinement}, and
\eqref{eq:filter-unit-support}.  Let
\(\{\Xc_j,w_{j,x}\}_{j\geq0}\) be a nested weighted sampling family, and
let \(s\geq0\) be an integer.
The shifted sampling family remains nested in the sense of $\Xc_{q+s}\subset \Xc_{q+1+s}$, $q\geq0$.
Define
\[
        \varphi_{J,x}^{(s)}
        :=
        w_{J+s,x}^{1/2}\varphi_{J,x},
        \qquad x\in\Xc_{J+s},
\]
and
\[
        \psi_{j,x}^{k,(s)}
        :=
        w_{j+1+s,x}^{1/2}\psi^k_{j,x},
        \qquad x\in\Xc_{j+1+s},\quad j\geq J,\quad k=1,\ldots,r .
\]
The family
\begin{equation}
\label{eq:oversampled-system}
        \mathrm{FS}_J^{(s)}
        :=
        \{\varphi_{J,x}^{(s)}:x\in\Xc_{J+s}\}
        \cup
        \{\psi_{j,x}^{k,(s)}:
        x\in\Xc_{j+1+s},\ j\geq J,\ k=1,\ldots,r\}
\end{equation}
is called the \emph{oversampled semi-discrete framelet system from nested Marcinkiewicz--Zygmund measures}
with \emph{oversampling shift} \(s\).
\end{definition}

\begin{theorem}[Near-tightness under nested oversampling]
\label{thm:oversampled-framelets}
Let $J_0\in\mathbb{N}_0$ and $J\geq J_0$.   Assume that $\Psi$ satisfies
\eqref{eq:spectral-lowpass-limit}, \eqref{eq:spectral-refinement}, and
\eqref{eq:filter-unit-support}.  Let
\(\{\Xc_j,w_{j,x}\}_{j\geq0}\) be the nested weighted family
constructed in Lemmas~\ref{lem:nested-quasi-uniform} and
\ref{lem:dyadic-weights} with a fixed parameter
\(0<\gamma<R_\Mc\), and form the oversampled nested semi-discrete
framelet system \eqref{eq:oversampled-system} from this family.
Let \(c\) be the constant in Proposition~\ref{prop:filbir-mhaskar}
with \(p=2\) and \(\kappa=1\), and set $C_{\Mc}:={2\gamma}/{c}$.
Let \(s\geq0\) be an integer such that $\rho_s:=C_{\Mc}2^{-s}<1$.
Then \(\mathrm{FS}_J^{(s)}\) is a
\(\rho_s\)-nearly tight frame, where
\(\rho_s=C_{\Mc}2^{-s}\).  In particular, these systems become
asymptotically tight as \(s\to\infty\).
\end{theorem}

\begin{proof}
We prove the estimate for a high-pass coefficient function; the 
low-pass term is treated similarly.  By Lemma~\ref{lem:coefficient-bandwidth},
$P_j^{\psi^k}\in\Pd_{2^j}$.
It is sampled on level $j+1+s$.  For the dyadic construction in
Lemma~\ref{lem:nested-quasi-uniform}, the mesh norm at this sampling
level satisfies $h_{\Xc_{j+1+s}}
        \leq
        \gamma 2^{-(j+1+s)}$.
Therefore, with $N_j:=2^j$, we have
$N_j h_{\Xc_{j+1+s}}
        \leq
        (\gamma/2)2^{-s}$.
Set $\eta_{j,s}^\psi:=({2}/{c})N_jh_{\Xc_{j+1+s}}$. Proposition~\ref{prop:filbir-mhaskar} with $p=2$ and
$\kappa=1$, together with 
$\eta_{j,s}^\psi\leq (C_{\Mc}/2)2^{-s}\leq\rho_s<1$,
gives a Marcinkiewicz--Zygmund inequality for $P_j^{\psi^k}$ with tolerance
$\eta_{j,s}^\psi$.  The constant is independent of the framelet scale $j$,
the high-pass channel $k$, and the starting scale $J$. By the definition of $\psi_{j,x}^{k,(s)}$, this Marcinkiewicz--Zygmund inequality becomes
\[
\begin{aligned}
        (1-\rho_s)\|P_j^{\psi^k}\|_{L^2(\mu)}^2
        \leq&
        (1-\eta_{j,s}^\psi)\|P_j^{\psi^k}\|_{L^2(\mu)}^2  
        \leq
        \sum_{x\in\Xc_{j+1+s}}
        |\langle f,\psi_{j,x}^{k,(s)}\rangle|^2 \leq
        (1+\eta_{j,s}^\psi)\|P_j^{\psi^k}\|_{L^2(\mu)}^2   \\
        \leq&
        (1+\rho_s)\|P_j^{\psi^k}\|_{L^2(\mu)}^2 .
\end{aligned}
\]
For $P_J^\varphi\in\Pd_{2^J}$, sampled on $\Xc_{J+s}$, we have
$2^Jh_{\Xc_{J+s}}\leq\gamma2^{-s}$; hence the same Marcinkiewicz--Zygmund argument gives
tolerance at most $C_\Mc2^{-s}=\rho_s$.  Summing the low-pass estimate and the high-pass estimates
over $J\leq j\leq L$ gives the desired two-sided estimate for the
finite partial coefficient energy.  Letting \(L\to\infty\), monotone
convergence applies to the discrete sums, while the continuous partial
energies converge to \(\|f\|_{L^2(\mu)}^2\) by
\eqref{eq:continuous-tight-frame-identity}.  This proves the claimed
frame bounds.
\end{proof}

\begin{remark}[Oversampling cost]
At high-pass scale \(j\), for example, the construction uses
\[
        |\Xc_{j+1+s}|\asymp \gamma^{-d}2^{(j+1+s)d},
        \qquad
        w_{j+1+s,x}\asymp \gamma^d2^{-(j+1+s)d}.
\]
Thus the improvement from \(1\pm O(1)\) to \(1\pm O(2^{-s})\) is
obtained by using about \(2^{sd}\) more centers than the basic
high-pass set \(\Xc_{j+1}\).
\end{remark}

\section{Fully discrete framelet transform and fast implementation}
\label{sec:fully-discrete-transform}

We now pass from semi-discrete framelets to a fully discrete transform.
We use the frame-theoretic terms \emph{analysis} and \emph{synthesis}.
In the filter-bank language, analysis is implemented by \emph{decomposition},
whereas synthesis gives \emph{reconstruction} directly only in the
perfect-reconstruction case.  The same nested weighted node sets as in the construction of semi-discrete framelets are
used throughout.

\subsection{Coefficients at the finest level}

Fix a coarsest level \(J\) and a finest scaling level \(K>J\).  Since the
generators are compactly supported on the nonnegative spectral axis, we define
the finite active index set $\Lambda_j^\alpha
        :=
        \left\{\ell\geq0:
        \widehat\alpha\!\left({\lambda_\ell}/{2^j}\right)
        \neq0\right\}$,
        $j\geq J$.
The support assumption \eqref{eq:filter-unit-support} implies
\(\Lambda_j^\alpha\subset\{\ell:\lambda_\ell\leq2^j\}\).  We assume, as is
the case for the filters used below, that
\(\Lambda_{j-1}^\alpha\subset\Lambda_j^\alpha\).  

For \(\widehat{\mathbf z}
=(z_\ell)_{\ell\in\Lambda_j^\alpha}
\in\mathbb C^{\Lambda_j^\alpha}\),
the weighted discrete Fourier transform $\mathbf F_j:\mathbb C^{\Lambda_j^\alpha}\longrightarrow
        \mathbb C^{\Xc_j}$ associated with $\Xc_j$ is defined by
\begin{equation}
\label{eq:coefficient-fourier-map}
        (\mathbf F_j\widehat{\mathbf z})(x)
        :=
        \sum_{\ell\in\Lambda_j^\alpha}
        z_\ell w_{j,x}^{1/2}u_\ell(x),\qquad x\in\Xc_j.
\end{equation}
Its adjoint discrete Fourier transform $\mathbf F_j^*:\mathbb C^{\Xc_j}\longrightarrow
        \mathbb C^{\Lambda_j^\alpha}$ is then given by
\begin{equation}
\label{eq:coefficient-adjoint-fourier-map}
        (\mathbf F_j^*\mathbf v)_\ell
        :=
        \sum_{x\in\Xc_j}
        v_xw_{j,x}^{1/2}\overline{u_\ell(x)},
        \qquad \ell\in\Lambda_j^\alpha.
\end{equation}
For the oversampled
construction we use instead
\(\mathbf F_j^{(s)}:\mathbb C^{\Lambda_j^\alpha}\to
\mathbb C^{\Xc_{j+s}}\), defined by replacing \(w_{j,x}\) and \(\Xc_j\)
in \eqref{eq:coefficient-fourier-map} by \(w_{j+s,x}\) and
\(\Xc_{j+s}\).  For the Marcinkiewicz--Zygmund measures considered below,
Proposition~\ref{prop:fourier-gram-bounds} shows that \(\mathbf F_j\)
is injective and well conditioned, although
\(\mathbf F_j^*\mathbf F_j\) need not be the identity.

\begin{proposition}[Gram bounds for the weighted Fourier map]
\label{prop:fourier-gram-bounds}
Suppose the level-\(j\) sampling rule has \(L^2\)
Marcinkiewicz--Zygmund constants \(1-\varepsilon_j\) and
\(1+\varepsilon_j\), where \(0\leq\varepsilon_j<1\).  Then on $\mathbb C^{\Lambda_j^\alpha}$,
\begin{equation}
\label{eq:fourier-gram-bounds}
        (1-\varepsilon_j)I
        \leq \mathbf F_j^*\mathbf F_j
        \leq (1+\varepsilon_j)I.
\end{equation}
\end{proposition}

\begin{proof}
For \(\widehat{\mathbf z}\in\mathbb C^{\Lambda_j^\alpha}\), set 
$P_{\widehat{\mathbf z}}
        :=\sum_{\ell\in\Lambda_j^\alpha}z_\ell u_\ell$.
Since \(\Lambda_j^\alpha\subseteq
\{\ell:\lambda_\ell\leq2^j\}\), we have
\(P_{\widehat{\mathbf z}}\in\Pd_{2^j}\).  Orthonormality gives
\(\|P_{\widehat{\mathbf z}}\|_{L^2(\mu)}^2
=\|\widehat{\mathbf z}\|_2^2\), while
$\|\mathbf F_j\widehat{\mathbf z}\|_2^2
        =\sum_{x\in\Xc_j}w_{j,x}
        |P_{\widehat{\mathbf z}}(x)|^2$.
The Marcinkiewicz--Zygmund inequality therefore yields
\[
        (1-\varepsilon_j)\|\widehat{\mathbf z}\|_2^2
        \leq\|\mathbf F_j\widehat{\mathbf z}\|_2^2
        \leq(1+\varepsilon_j)\|\widehat{\mathbf z}\|_2^2.
\]
Since
\(\|\mathbf F_j\widehat{\mathbf z}\|_2^2
=\langle\mathbf F_j^*\mathbf F_j\widehat{\mathbf z},
\widehat{\mathbf z}\rangle\), this is equivalent to
\eqref{eq:fourier-gram-bounds}.
\end{proof}

For \(f\in L^2(\Mc,\mu)\), to start the finite transform, we need the spectral vector
\begin{equation*}
        \widehat{\mathbf v}_K
        =
        \left(
        \widehat f_\ell\,
        \overline{\widehat\alpha(\lambda_\ell/2^K)}
        \right)_{\ell\in\Lambda_K^\alpha}
\end{equation*}
of the finest scaling coefficient function
\(v_K(y):=\langle f,\varphi_{K,y}\rangle\).  If only
the weighted nodal values $\mathbf v_K$ of $f$ are available, we compute a
spectral fit by
\[
        \widehat{\mathbf v}_K
        =
        \operatorname*{argmin}_{\mathbf z\in\mathbb C^{\Lambda_K^\alpha}}
        \|\mathbf F_K\mathbf z-\mathbf v_K\|_2^2,
\]
equivalently, $        \mathbf F_K^*\mathbf F_K\widehat{\mathbf v}_K
        =\mathbf F_K^*\mathbf v_K$.
Algorithm~\ref{alg:finest-level-cg} solves it by a conjugate-gradient
iteration.  Here and below, \(\epsilon_{\mathrm{mach}}\) denotes machine
precision.

\begin{algorithm}[htbp]
\caption{Finest-level scaling coefficient computation}
\label{alg:finest-level-cg}
\KwIn{A weighted finest-level scaling sequence \(\mathbf v_K\), operators
\(\mathbf F_K\) and \(\mathbf F_K^*\), tolerance \(\tau>0\), and maximum
iteration number \(M_{\max}\)}
\KwOut{Approximation to
\(\widehat{\mathbf v}_K\in\mathbb C^{\Lambda_K^\alpha}\)}
\BlankLine
\(\mathbf b\leftarrow\mathbf F_K^*\mathbf v_K\),
\(\mathbf z\leftarrow0\)\;
\(\mathbf r\leftarrow\mathbf b\),
\(\mathbf p\leftarrow\mathbf r\)\;
\For{\(m=0,1,\ldots,M_{\max}-1\)}{
        \If{\(\|\mathbf r\|/\max\{\|\mathbf b\|,\epsilon_{\mathrm{mach}}\}\leq\tau\)}{
                \KwRet{\(\widehat{\mathbf v}_K\leftarrow\mathbf z\)}
        }
        \(\mathbf q\leftarrow\mathbf F_K^*\mathbf F_K\mathbf p\)\;
        \(\theta_m\leftarrow\langle\mathbf r,\mathbf r\rangle/
        \langle\mathbf p,\mathbf q\rangle\)\;
        \(\mathbf z\leftarrow\mathbf z+\theta_m\mathbf p\)\;
        \(\mathbf r_{\rm new}\leftarrow\mathbf r-\theta_m\mathbf q\)\;
        \(\delta_m\leftarrow
        \langle\mathbf r_{\rm new},\mathbf r_{\rm new}\rangle/
        \langle\mathbf r,\mathbf r\rangle\)\;
        \(\mathbf p\leftarrow\mathbf r_{\rm new}+\delta_m\mathbf p\),
        \(\mathbf r\leftarrow\mathbf r_{\rm new}\)\;
}
\(\widehat{\mathbf v}_K\leftarrow\mathbf z\)\;
\end{algorithm}

\begin{remark}\label{rem:product}
Let \(C_{\mathrm{prod}}\) be a product-bandwidth constant as in
Remark~\ref{rem:products-diffusion-polynomials}.  Suppose that the
level-\(K\) sampling rule is exact for polynomials of degree up to
\(C_{\mathrm{prod}}2^K\), and that the corresponding product condition holds at
level \(K\).
For \(\widehat{\mathbf z},\widehat{\mathbf y}\in
\mathbb C^{\Lambda_K^\alpha}\), set
\(P_{\widehat{\mathbf z}}:=\sum_{\ell\in\Lambda_K^\alpha}z_\ell u_\ell\)
and
\(P_{\widehat{\mathbf y}}:=\sum_{\ell\in\Lambda_K^\alpha}y_\ell u_\ell\).
These functions belong to \(\Pd_{2^K}\).
Hence exactness applied to \(P_{\widehat{\mathbf z}}\overline{P_{\widehat{\mathbf y}}}\)
gives
\[
 \langle\mathbf F_K\widehat{\mathbf z},
 \mathbf F_K\widehat{\mathbf y}\rangle
 =\sum_{x\in\Xc_K}w_{K,x}P_{\widehat{\mathbf z}}(x)
 \overline{P_{\widehat{\mathbf y}}(x)}
 =\langle\widehat{\mathbf z},\widehat{\mathbf y}\rangle.
\]
Thus \(\mathbf F_K^*\mathbf F_K=I\). Consequently,
Algorithm~\ref{alg:finest-level-cg} is unnecessary, and the finest-level
coefficients can be computed directly as
$\widehat{\mathbf v}_K=\mathbf F_K^*\mathbf v_K$
using the adjoint discrete Fourier transform. This recovers the setting considered in~\cite{wang2020tight}.
\end{remark}
 
\subsection{Filter banks}

The continuous framelet system $\mathrm{CFS}_J(\Psi)$ was described in terms of 
generators $\Psi=\{\alpha;\beta^1,\ldots,\beta^r\}$.
For an actual multilevel algorithm, one also needs a filter bank
\[
        \Theta=\{a;b^1,\ldots,b^r\},
\]
where $a,b^1,\ldots,b^r\in\ell_1(\mathbb Z)$. Here $a$ is the
low-pass filter and $b^1,\ldots,b^r$ are high-pass filters. We write their symbols, using the same
$2\pi$-Fourier convention as above, as $\widehat a(\xi)=\sum_{m\in\mathbb Z}a_me^{-2\pi im\xi}$ and 
$\widehat{b^k}(\xi)=\sum_{m\in\mathbb Z}b_m^ke^{-2\pi im\xi}$.
The role of the filter bank is to connect adjacent scales. The generators and the filters are related
by the refinement relations
\begin{equation}
\label{eq:filter-refinement-equations}
        \widehat\alpha(2\xi)
        =\widehat a(\xi)\widehat\alpha(\xi),
        \qquad
        \widehat{\beta^k}(2\xi)
        =\widehat{b^k}(\xi)\widehat\alpha(\xi),
        \quad k=1,\ldots,r.
\end{equation}
Thus $\alpha$ and $\beta^k$ describe the framelet atoms at absolute scales, while $a$ and $b^k$ describe the transition
from one scale to the next. For example,
\begin{equation}
\label{eq:scaled-filter-refinement-equations}
        \widehat\alpha\left(\frac{\lambda_\ell}{2^{j-1}}\right)
        =\widehat a\left(\frac{\lambda_\ell}{2^j}\right)
          \widehat\alpha\left(\frac{\lambda_\ell}{2^j}\right),\qquad
        \widehat{\beta^k}\left(\frac{\lambda_\ell}{2^{j-1}}\right)
        =\widehat{b^k}\left(\frac{\lambda_\ell}{2^j}\right)
          \widehat\alpha\left(\frac{\lambda_\ell}{2^j}\right).
\end{equation}
We assume that the filters split the energy of every active coefficient:
\begin{equation}
\label{eq:filter-bank-partition}
        \left|\widehat a\left(\frac{\lambda_\ell}{2^j}\right)\right|^2
        +\sum_{k=1}^r\left|\widehat{b^k}\left(\frac{\lambda_\ell}{2^j}\right)\right|^2
        =1,
        \qquad \ell\in\Lambda_j^\alpha,\quad j>J.
\end{equation}
Together, these assumptions provide the algebraic basis of the fully
discrete transform.

\begin{proposition}[Generator conditions from a filter bank]
\label{prop:filter-bank-implies-generator-conditions}
Let \(\Psi=\{\alpha;\beta^1,\ldots,\beta^r\}\subset L_1(\mathbb R)\)
and the filter bank \(\Theta=\{a;b^1,\ldots,b^r\}\subset \ell_1(\mathbb Z)\).  Suppose
that the refinement equations \eqref{eq:filter-refinement-equations}
hold for all \(\xi\in\mathbb R\), that the active spectral partition
\eqref{eq:filter-bank-partition} holds, and that
\[
        |\widehat\alpha(0)|=1,
        \qquad
        \supp\widehat\alpha\cap[0,\infty)\subset [0,1/2].
\]
Then \(\Psi\) satisfies the low-pass limit
\eqref{eq:spectral-lowpass-limit}, the  identity
\eqref{eq:spectral-refinement}, and the support condition
\eqref{eq:filter-unit-support}.
\end{proposition}

\begin{proof}
Since \(\alpha\in L_1(\mathbb R)\), its Fourier transform is continuous.
For each fixed \(\ell\geq0\), \(\lambda_\ell/2^j\to0\) as
\(j\to\infty\).  Hence
\[
        \lim_{j\to\infty}
        \left|
        \widehat\alpha\!\left(\frac{\lambda_\ell}{2^j}\right)
        \right|
        =
        |\widehat\alpha(0)|
        =
        1,
\]
which is \eqref{eq:spectral-lowpass-limit}.
Next fix \(j\geq J\) and \(\ell\geq0\), and set
\(\xi=\lambda_\ell/2^{j+1}\).  The refinement equations give
\[
        \widehat\alpha\!\left(\frac{\lambda_\ell}{2^j}\right)
        =
        \widehat a(\xi)
        \widehat\alpha\!\left(\frac{\lambda_\ell}{2^{j+1}}\right),
        \qquad
        \widehat{\beta^k}\!\left(\frac{\lambda_\ell}{2^j}\right)
        =
        \widehat{b^k}(\xi)
        \widehat\alpha\!\left(\frac{\lambda_\ell}{2^{j+1}}\right).
\]
If
\(\widehat\alpha(\lambda_\ell/2^{j+1})=0\), then all terms in
\eqref{eq:spectral-refinement} are zero.  Otherwise
\(\ell\in\Lambda_{j+1}^\alpha\), and \eqref{eq:filter-bank-partition},
applied at level \(j+1\), yields
\[ 
        \left|
        \widehat\alpha\!\left(\frac{\lambda_\ell}{2^j}\right)
        \right|^2
        +
        \sum_{k=1}^r
        \left|
        \widehat{\beta^k}\!\left(\frac{\lambda_\ell}{2^j}\right)
        \right|^2  =
        \left(
        |\widehat a(\xi)|^2
        +
        \sum_{k=1}^r|\widehat{b^k}(\xi)|^2
        \right)
        \left|
        \widehat\alpha\!\left(\frac{\lambda_\ell}{2^{j+1}}\right)
        \right|^2
        =
        \left|
        \widehat\alpha\!\left(\frac{\lambda_\ell}{2^{j+1}}\right)
        \right|^2 .
\]
This proves \eqref{eq:spectral-refinement}. Finally, the assumed support of \(\widehat\alpha\) gives
\(\supp\widehat\alpha\cap[0,\infty)\subset[0,1]\).  For each \(k\),
the second refinement equation implies
$\widehat{\beta^k}(\eta)
        =
        \widehat{b^k}(\eta/2)\widehat\alpha(\eta/2),
        \eta\in\mathbb R$.
Thus, on the nonnegative axis,
\(\widehat{\beta^k}(\eta)\neq0\) can occur only when
\(\eta/2\in\supp\widehat\alpha\cap[0,\infty)\subset[0,1/2]\), that is,
only when \(\eta\in[0,1]\).  Hence
\(\supp\widehat{\beta^k}\cap[0,\infty)\subset[0,1]\) for every \(k\),
and \eqref{eq:filter-unit-support} follows.
\end{proof}

\subsection{Analysis}
We next define the level-transition maps used by the filter bank. Let $\mathbf{v}_j=\mathbf{F}_j \widehat{\mathbf{v}}_j$ be a weighted nodal sequence on $\Xc_j$, where $\widehat{\mathbf{v}}_j=\left(\widehat{v}_{j, \ell}\right)_{\ell \in \Lambda_j^\alpha}$. For a filter $h \in \ell_1(\mathbb{Z})$, we define the discrete convolution of $\mathbf{v}_j$ with $h$ by
\[
\left(\mathbf{v}_j *_j h\right)_x:=\sum_{\ell \in \Lambda_j^\alpha} \widehat{v}_{j, \ell} \widehat{h}\left(\frac{\lambda_{\ell}}{2^j}\right) w_{j, x}^{1 / 2} u_{\ell}(x), \quad x \in \Xc_j .
\]
The downsampling operator is defined in the same spectral language: 
\[
\left(\mathbf{v}_j \downarrow_j\right)_x:=\sum_{\ell \in \Lambda_j^\alpha} \widehat{v}_{j, \ell} w_{j-1, x}^{1 / 2} u_{\ell}(x), \quad x \in \Xc_{j-1} .
\]
Thus $\downarrow_j$ is not a deletion of entries; it is a change of weighted sampling nodes from $\Xc_j$ to $\Xc_{j-1}$. Conversely, the upsampling operator from level $j-1$ to level $j$ is defined by
\[
\left(\mathbf{v}_{j-1} \uparrow_j\right)_x:=\sum_{\ell \in \Lambda_{j-1}^\alpha} \widehat{v}_{j-1, \ell} w_{j, x}^{1 / 2} u_{\ell}(x), \quad x \in \Xc_j .
\]

The next result explains why the analysis step is implemented by the conjugate filters. For a filter $h$, let $h^*$ be characterized by $\widehat{h^*}(\xi)=\overline{\widehat{h}(\xi)}$ on real $\xi$.

\begin{theorem}[One-step coefficient decomposition]
\label{thm:one-step-filter-bank-analysis}
Suppose that the generators and filters
satisfy the conditions in Proposition \ref{prop:filter-bank-implies-generator-conditions}. Fix
$f \in L^2(\Mc, \mu)$ and a level $j>J$. Define the coefficient functions
\[
v_j(y):=\left\langle f, \varphi_{j, y}\right\rangle, \quad v_{j-1}(y):=\left\langle f, \varphi_{j-1, y}\right\rangle, \qquad w_{j-1}^k(y):=\langle f, \psi_{j-1, y}^k\rangle, \quad k=1, \ldots, r .
\]
Let $\mathbf{v}_j$ and $\mathbf{v}_{j-1}$ be the weighted nodal sequences of $v_j$ and $v_{j-1}$ on $\Xc_j$ and $\Xc_{j-1}$, respectively, and let $\mathbf{w}_{j-1}^k$ be the weighted nodal sequence of $w_{j-1}^k$ on $\Xc_j$. Then
\[
\mathbf{v}_{j-1}=\left(\mathbf{v}_j *_j a^*\right) \downarrow_j, \qquad \mathbf{w}_{j-1}^k=\mathbf{v}_j *_j\left(b^k\right)^*, \quad k=1, \ldots, r .
\]
\end{theorem}

\begin{proof}
By the coefficient expansion \eqref{eq:framelet-coefficient-expansions},
the spectral coefficients of $v_j$ are 
$\widehat v_{j,\ell}
        =
        \widehat f_\ell\,
        \overline{\widehat\alpha\left({\lambda_\ell}/{2^j}\right)}$,
$\ell\in\Lambda_j^\alpha$.
The scaled refinement equations \eqref{eq:scaled-filter-refinement-equations}
therefore give, for every $\ell\in\Lambda_j^\alpha$,
\[
 \widehat v_{j,\ell}\,
 \overline{\widehat a\left(\frac{\lambda_\ell}{2^j}\right)}=
 \widehat f_\ell\,
 \overline{\widehat\alpha\left(\frac{\lambda_\ell}{2^{j-1}}\right)},\qquad
 \widehat v_{j,\ell}\,
 \overline{\widehat{b^k}\left(\frac{\lambda_\ell}{2^j}\right)}=
 \widehat f_\ell\,
 \overline{\widehat{\beta^k}\left(\frac{\lambda_\ell}{2^{j-1}}\right)},
 \qquad k=1,\ldots,r.
\]
For $\ell\notin\Lambda_j^\alpha$, the scaled refinement equations show
that both coarse- and high-pass coefficients on the right vanish as well.
Thus these identities determine all the spectral coefficients of the
corresponding coefficient functions $v_{j-1}(y)=\langle f,\varphi_{j-1,y}\rangle$ and $w^k_{j-1}(y) = \langle f,\psi^k_{j-1,y}\rangle$.  Evaluating the first one on
$\Xc_{j-1}$ with the weights $w_{j-1,x}$ yields $\left(\mathbf v_j*_j a^*\right)\downarrow_j
        =\mathbf v_{j-1}$,
while evaluating the second identity on $\Xc_j$ with the weights
$w_{j,x}$ yields $\mathbf v_j*_j(b^k)^*=\mathbf w_{j-1}^k$ for $k=1,\ldots,r$.
\end{proof}

Starting from the finest-level vector $\widehat{\mathbf v}_K$ computed by
Algorithm~\ref{alg:finest-level-cg}, we repeat the one-step decomposition
for $j=K,K-1,\ldots,J+1$.  The restriction
\(\mathcal R_j^{j-1}:\mathbb C^{\Lambda_j^\alpha}\to
\mathbb C^{\Lambda_{j-1}^\alpha}\) is taken onto
\(\Lambda_{j-1}^\alpha\), that is,
\[
 \bigl(\mathcal R_j^{j-1}\widehat{\mathbf z}\bigr)_\ell=z_\ell,
 \qquad \ell\in\Lambda_{j-1}^\alpha.
\]
Algorithm~\ref{alg:multi-level-fmt-analysis} first performs the spectral filtering and then
stores each branch on its prescribed weighted node set.

\begin{algorithm}[htbp]
\caption{Multi-level coefficient analysis}
\label{alg:multi-level-fmt-analysis}
\KwIn{Finest-level scaling vector
\(\widehat{\mathbf v}_K\in\mathbb C^{\Lambda_K^\alpha}\), filters
\(a,b^1,\ldots,b^r\), and levels \(J<K\)}
\KwOut{Coarse sequence \(\mathbf v_J\) and detail sequences
\(\mathbf w_{j-1}^k\), \(j=J+1,\ldots,K\), \(k=1,\ldots,r\)}
\BlankLine
\For{\(j=K,K-1,\ldots,J+1\)}{
        \For{\(k=1,\ldots,r\)}{
                \(\widehat{\mathbf w}_{j-1}^k
                \leftarrow
                \widehat{\mathbf v}_j
                \overline{\widehat{b^k}(2^{-j}\lambda)}\)\;
                \(\mathbf w_{j-1}^k
                \leftarrow\mathbf F_j
                \widehat{\mathbf w}_{j-1}^k\)\;
        }
        \(\widehat{\mathbf v}_{j-1}
        \leftarrow\mathcal R_j^{j-1}
        \bigl(\widehat{\mathbf v}_j
        \overline{\widehat a(2^{-j}\lambda)}\bigr)\)\;
}
\(\mathbf v_J\leftarrow\mathbf F_J\widehat{\mathbf v}_J\)\;
\end{algorithm}

Let \(\mathcal A_{J,K}:\mathbb C^{\Lambda_K^\alpha}
 \rightarrow
 \mathbb C^{\Xc_J}\times
 \prod_{j=J+1}^K\bigl(\mathbb C^{\Xc_j}\bigr)^r\) be the linear map implemented by
Algorithm~\ref{alg:multi-level-fmt-analysis}, namely
\[
 \mathcal A_{J,K}\widehat{\mathbf v}_K=
 \left(\mathbf v_J,\{\mathbf w_{j-1}^k\}\right).
\]
Equip the domain and codomain with their standard Euclidean inner
products.  The adjoint \(\mathcal A_{J,K}^*\) is therefore the unique
linear map satisfying
\[
 \langle\mathcal A_{J,K}\widehat{\mathbf v}_K,\mathbf c\rangle
 =\langle\widehat{\mathbf v}_K,
 \mathcal A_{J,K}^*\mathbf c\rangle
\]
for every \(\widehat{\mathbf v}_K\) and coefficient collection
\(\mathbf c\).
For an oversampling shift \(s\geq0\), the operator
\(\mathcal A_{J,K}^{(s)}\) is defined by the same spectral filtering
steps, with \(\mathbf F_J\) and \(\mathbf F_j\) replaced by
\(\mathbf F_J^{(s)}\) and \(\mathbf F_j^{(s)}\), respectively.

\subsection{Synthesis}
We now synthesize a finest-level scaling vector from a coefficient
collection
\[
        \mathbf c=\left(\mathbf c_J,
        \{\mathbf c_{j-1}^k:j=J+1,\ldots,K,\ k=1,\ldots,r\}\right),
\]
using the adjoint filter bank. Define the zero-extension
map \(\mathcal E_{j-1}^j:\mathbb C^{\Lambda_{j-1}^\alpha}\to
\mathbb C^{\Lambda_j^\alpha}\) by
\[
 \bigl(\mathcal E_{j-1}^j\widehat{\mathbf z}\bigr)_\ell
 =\begin{cases}
 z_\ell,&\ell\in\Lambda_{j-1}^\alpha,\\
 0,&\ell\in\Lambda_j^\alpha\setminus\Lambda_{j-1}^\alpha.
 \end{cases}
\] 
It is the Euclidean adjoint of the restriction operator
\(\mathcal R_j^{j-1}\) in the analysis step.
The recursion for the Euclidean adjoint
\(\mathcal A_{J,K}^*\) is
\[
        \widehat{\mathbf z}_J:=\mathbf F_J^*\mathbf c_J,\qquad
        \widehat{\mathbf z}_j
        :={\mathcal E}_{j-1}^j\widehat{\mathbf z}_{j-1}
        \widehat a(2^{-j}\lambda_{\cdot})
        +\sum_{k=1}^r
        \bigl(\mathbf F_j^*\mathbf c_{j-1}^k\bigr)
        \widehat{b^k}(2^{-j}\lambda_{\cdot}),
        \quad j=J+1,\ldots,K.
\]
This recursion is obtained by taking the Euclidean adjoint of each step
of Algorithm~\ref{alg:multi-level-fmt-analysis}: evaluation by
\(\mathbf F_j\) becomes \(\mathbf F_j^*\), while multiplication by
the conjugate filters becomes multiplication by \(\widehat a\) and
\(\widehat{b^k}\). It is described in the following
Algorithm~\ref{alg:multi-level-fmt-synthesis}, computing
\(\mathcal A_{J,K}^*\mathbf c\).

\begin{algorithm}[htbp]
\caption{Multi-level coefficient one-pass synthesis}
\label{alg:multi-level-fmt-synthesis}
\KwIn{Coarse sequence \(\mathbf c_J\), detail sequences
\(\mathbf c_{j-1}^k\), filters \(a,b^1,\ldots,b^r\), and levels \(J<K\)}
\KwOut{One-pass finest-level spectral vector
\(\widehat{\mathbf z}_K\)}
\BlankLine
\(\widehat{\mathbf z}_J
\leftarrow\mathbf F_J^*\mathbf c_J\)\;
\For{\(j=J+1,J+2,\ldots,K\)}{
        \(\widehat{\mathbf z}_j
        \leftarrow
        \mathcal E_{j-1}^j
        \widehat{\mathbf z}_{j-1}
        \widehat a(2^{-j}\lambda_{\cdot})\)\;
        \For{\(k=1,\ldots,r\)}{
                \(\widehat{\mathbf e}_{j-1}^k
                \leftarrow\mathbf F_j^*\mathbf c_{j-1}^k\)\;
                \(\widehat{\mathbf z}_j
                \leftarrow
                \widehat{\mathbf z}_j+
                \widehat{\mathbf e}_{j-1}^k
                \widehat{b^k}(2^{-j}\lambda_{\cdot})\)\;
        }
}
\end{algorithm}

When the sampling rules have sufficient algebraic exactness, the adjoint
filter-bank pass is already an exact inverse; no iterative correction is
needed.  More precisely, we have the following result. 

\begin{proposition}[Exact one-pass synthesis]
\label{prop:exact-one-pass-synthesis}
Suppose that the refinement relations
\eqref{eq:filter-refinement-equations} and the filter-bank partition
\eqref{eq:filter-bank-partition} hold.  Let \(C_{\mathrm{prod}}\) be a
product-bandwidth constant as in
Remark~\ref{rem:products-diffusion-polynomials}.  Assume, for every
\(j=J,\ldots,K\), that the level-\(j\) sampling rule is exact for
polynomials of degree up to \(C_{\mathrm{prod}}2^j\), and that the corresponding
product condition holds.  Then
\begin{equation}
\label{eq:exact-coefficient-reconstruction}
        \mathcal A_{J,K}^*\mathcal A_{J,K}=I
        \quad\text{on }\mathbb C^{\Lambda_K^\alpha}.
\end{equation}
Consequently, if
\(\mathbf c=\mathcal A_{J,K}\widehat{\mathbf v}_K\), then
Algorithm~\ref{alg:multi-level-fmt-synthesis} returns
\(\widehat{\mathbf v}_K\) after one pass.
\end{proposition}

\begin{proof}
Recall that in this case, \(\mathbf F_j^*\mathbf F_j=I\) on
\(\mathbb C^{\Lambda_j^\alpha}\) for every \(j\), including the
coarsest branch at level \(J\); see the argument in Remark \ref{rem:product}. We consider one spectral filter-bank step.  The low-pass branch has
coefficients \(\widehat{\mathbf v}_j\overline{\widehat a}\), restricted
to \(\Lambda_{j-1}^\alpha\), and the $k$th detail branch has coefficients
\(\widehat{\mathbf v}_j\overline{\widehat{b^k}}\).  The refinement
relation makes the low-pass product vanish outside
\(\Lambda_{j-1}^\alpha\).  Therefore the partition identity gives
\[
 \|\widehat{\mathbf v}_j\|_2^2
 =\|\widehat{\mathbf v}_{j-1}\|_2^2+
 \sum_{k=1}^r\|\widehat{\mathbf w}_{j-1}^k\|_2^2.
\]
Iterating this identity shows that the spectral filter-bank map
\[
 \mathcal W_{J,K}:
 \mathbb C^{\Lambda_K^\alpha}
 \longrightarrow
 \mathbb C^{\Lambda_J^\alpha}
 \oplus
 \bigoplus_{j=J+1}^K
 \bigl(\mathbb C^{\Lambda_j^\alpha}\bigr)^r,
 \qquad
 \widehat{\mathbf v}_K
 \longmapsto
 \left(\widehat{\mathbf v}_J,
 \{\widehat{\mathbf w}_{j-1}^k\}\right),
\]
is an isometry, so \(\mathcal W_{J,K}^*\mathcal W_{J,K}=I\).
Let
\[
 \mathcal D_{J,K}:
 \mathbb C^{\Lambda_J^\alpha}
 \oplus
 \bigoplus_{j=J+1}^K
 \bigl(\mathbb C^{\Lambda_j^\alpha}\bigr)^r
 \longrightarrow
 \mathbb C^{\Xc_J}
 \oplus
 \bigoplus_{j=J+1}^K
 \bigl(\mathbb C^{\Xc_j}\bigr)^r
\]
be the block-diagonal map that applies \(\mathbf F_J\) to the coarse
branch and \(\mathbf F_j\) to each level-\(j\) detail branch.  The fact \(\mathbf F_j^*\mathbf F_j=I\)
 shows that every block is an isometry, hence
\(\mathcal D_{J,K}^*\mathcal D_{J,K}=I\).

Since the analysis operator
implemented by Algorithm~\ref{alg:multi-level-fmt-analysis} factors as
\(\mathcal A_{J,K}=\mathcal D_{J,K}\mathcal W_{J,K}\), we obtain
\[
        \mathcal A_{J,K}^*\mathcal A_{J,K}
        =\mathcal W_{J,K}^*\mathcal D_{J,K}^*
        \mathcal D_{J,K}\mathcal W_{J,K}
        =\mathcal W_{J,K}^*\mathcal W_{J,K}
        =I
        \quad\text{on }\mathbb C^{\Lambda_K^\alpha}.
\]
The conclusion follows because the one-pass algorithm computes
\(\mathcal A_{J,K}^*\mathbf c\).
\end{proof}

For Marcinkiewicz--Zygmund sampling, the weighted Fourier maps are
well conditioned but generally not isometries.  Hence one-pass synthesis
is still the adjoint operation, but
\(\mathcal A_{J,K}^*\mathcal A_{J,K}\neq I\) in general.  To recover
the canonical coefficient vector, we instead solve
\begin{equation*}
        \widehat{\mathbf u}_K
        =\operatorname*{argmin}_{\widehat{\mathbf z}\in
        \mathbb C^{\Lambda_K^\alpha}}
        \frac12\|\mathcal A_{J,K}\widehat{\mathbf z}-\mathbf c\|_2^2.
\end{equation*}
Its normal equation, involving the coefficient frame operator $\mathbf S_{J,K}:=\mathcal A_{J,K}^*\mathcal A_{J,K}$, is given by
\begin{equation*}
        \mathbf S_{J,K}\widehat{\mathbf u}_K
        =\mathcal A_{J,K}^*\mathbf c.
\end{equation*}
Algorithm~\ref{alg:canonical-cg-reconstruction} applies conjugate
gradients to this normal equation on the full coefficient space.

\begin{algorithm}[htbp]
\caption{Coefficient-space frame-operator CG reconstruction}
\label{alg:canonical-cg-reconstruction}
\KwIn{Coefficient collection \(\mathbf c\), tolerance \(\tau>0\), and
maximum iteration number \(M_{\max}\)}
\KwOut{Canonical finest-level coefficient vector
\(\widehat{\mathbf u}_K\in\mathbb C^{\Lambda_K^\alpha}\)}
\BlankLine
\(\mathbf b\leftarrow\mathcal A_{J,K}^*\mathbf c\),
\(\widehat{\mathbf u}_K\leftarrow0\)\;
\(\mathbf r\leftarrow\mathbf b\), \(\mathbf p\leftarrow\mathbf r\)\;
\For{\(m=0,1,\ldots,M_{\max}-1\)}{
        \If{\(\|\mathbf r\|/\max\{\|\mathbf b\|,\epsilon_{\mathrm{mach}}\}\leq\tau\)}{
                \KwRet{\(\widehat{\mathbf u}_K\)}
        }
        \(\mathbf q\leftarrow
        \mathcal A_{J,K}^*\mathcal A_{J,K}\mathbf p\)\;
        \(\theta_m\leftarrow\langle\mathbf r,\mathbf r\rangle/
        \langle\mathbf p,\mathbf q\rangle\)\;
        \(\widehat{\mathbf u}_K\leftarrow
        \widehat{\mathbf u}_K+\theta_m\mathbf p\)\;
        \(\mathbf r_{\rm new}\leftarrow\mathbf r-\theta_m\mathbf q\)\;
        \(\delta_m\leftarrow
        \langle\mathbf r_{\rm new},\mathbf r_{\rm new}\rangle/
        \langle\mathbf r,\mathbf r\rangle\)\;
        \(\mathbf p\leftarrow\mathbf r_{\rm new}+\delta_m\mathbf p\),
        \(\mathbf r\leftarrow\mathbf r_{\rm new}\)\;
}
\end{algorithm}

The Marcinkiewicz--Zygmund bounds do more than ensure that the least-squares problem is
well posed. They quantify the conditioning of its normal equation and
therefore the convergence of conjugate gradients.

\begin{proposition}[Finite-level near-tightness and CG reconstruction]\label{prop:finite-coefficient-frame-bounds}
Assume that the refinement relations
\eqref{eq:filter-refinement-equations} hold and that
\eqref{eq:filter-bank-partition} holds for \(j=J+1,\ldots,K\).
\begin{enumerate}[label=\textup{(\roman*)}]
\item For the fixed-tolerance nested Marcinkiewicz--Zygmund measures of
Theorem~\ref{thm:fixed-tolerance-mz}, with $p=2$ and tolerance
$0<\eta<1$,
\[
 (1-\eta)\|\widehat{\mathbf z}\|_2^2
 \leq\|\mathcal A_{J,K}\widehat{\mathbf z}\|_2^2
 \leq(1+\eta)\|\widehat{\mathbf z}\|_2^2
 \quad\text{for all }\widehat{\mathbf z}\in\mathbb C^{\Lambda_K^\alpha}.
\]
\item In the oversampled setting of
Theorem~\ref{thm:oversampled-framelets}, with
\(\rho_s=C_\Mc2^{-s}<1\), the corresponding operator
\(\mathcal A_{J,K}^{(s)}\) satisfies the same inequality with
\(\eta\) replaced by \(\rho_s\).
\end{enumerate}
Both conclusions remain valid after restriction to a spectral subspace.
In either case, let \(\mathcal T=\mathcal A_{J,K}\) in \textup{(i)}
and \(\mathcal T=\mathcal A_{J,K}^{(s)}\) in \textup{(ii)}, and write
\(\varepsilon=\eta\) or \(\varepsilon=\rho_s\), respectively.  Then
\(\mathbf S_{\mathcal T}:=\mathcal T^*\mathcal T\) is a positive-definite
self-adjoint operator on \(\mathbb C^{\Lambda_K^\alpha}\).
Let
\(\lambda_{\min}(\mathbf S_{\mathcal T})\) and
\(\lambda_{\max}(\mathbf S_{\mathcal T})\)
denote its smallest and largest eigenvalues, respectively.  Its spectral
condition number is
\begin{equation}
\label{eq:coefficient-frame-condition-number}
        \kappa(\mathbf S_{\mathcal T})
        :=
        \frac{\lambda_{\max}(\mathbf S_{\mathcal T})}
             {\lambda_{\min}(\mathbf S_{\mathcal T})}
        \leq
        \frac{1+\varepsilon}{1-\varepsilon}.
\end{equation}
Denote by
\(\|\widehat{\mathbf z}\|_{\mathbf S_{\mathcal T}}
:=\langle\mathbf S_{\mathcal T}\widehat{\mathbf z},
\widehat{\mathbf z}\rangle^{1/2}\).
If \(\widehat{\mathbf u}_K\) solves
\(\mathbf S_{\mathcal T}\widehat{\mathbf u}_K=\mathcal T^*\mathbf c\)
and \(\widehat{\mathbf u}_K^{(m)}\) is the \(m\)th CG iterate (using
\(\mathcal T\) in place of \(\mathcal A_{J,K}\) in
Algorithm~\ref{alg:canonical-cg-reconstruction}), then
\begin{equation}
\label{eq:coefficient-cg-convergence}
 \|\widehat{\mathbf u}_K^{(m)}-\widehat{\mathbf u}_K\|_{\mathbf S_{\mathcal T}}
 \leq 2\left(
 \frac{\sqrt{\kappa(\mathbf S_{\mathcal T})}-1}
 {\sqrt{\kappa(\mathbf S_{\mathcal T})}+1}\right)^m
 \|\widehat{\mathbf u}_K^{(0)}-\widehat{\mathbf u}_K\|_{\mathbf S_{\mathcal T}}.
\end{equation}
\end{proposition}

\begin{proof}
For \(\widehat{\mathbf z}\in\mathbb C^{\Lambda_K^\alpha}\), iterating
the one-level filter-bank energy identity gives
\[
 \|\widehat{\mathbf z}\|_2^2
 =\|\widehat{\mathbf v}_J\|_2^2
 +\sum_{j=J+1}^K\sum_{k=1}^r
 \|\widehat{\mathbf w}_{j-1}^k\|_2^2.
\]
The polynomial represented by \(\widehat{\mathbf v}_J\) belongs to
\(\Pd_{2^J}\), while the polynomial represented by
\(\widehat{\mathbf w}_{j-1}^k\) belongs to \(\Pd_{2^j}\).  In case
\textup{(i)}, the levelwise Marcinkiewicz--Zygmund inequalities therefore give
\[
 (1-\eta)\|\widehat{\mathbf z}\|_2^2
 \leq\|\mathcal A_{J,K}\widehat{\mathbf z}\|_2^2
 \leq(1+\eta)\|\widehat{\mathbf z}\|_2^2
\]
after summing over the coarse and detail branches.  In case \textup{(ii)},
the same argument samples a degree-\(2^j\) branch on \(\Xc_{j+s}\).
Theorem~\ref{thm:fixed-degree-refinement} gives the uniform error
\(\rho_s=C_\Mc2^{-s}\), and hence the asserted bounds for
\(\mathcal A_{J,K}^{(s)}\).  Restricting either operator to a spectral
subspace preserves these inequalities.
The resulting frame bounds imply
\((1-\varepsilon)I\leq\mathbf S_{\mathcal T}\leq(1+\varepsilon)I\), which
proves positive definiteness and \eqref{eq:coefficient-frame-condition-number}.
Finally, \eqref{eq:coefficient-cg-convergence} is the standard
conjugate-gradient estimate for a positive-definite linear system; see, e.g., \cite[Lecture 38]{MR1444820}.
\end{proof}

\subsection{Fast algorithms and computational cost}

Fast algorithms for Fourier maps and their adjoints exist on many
standard manifolds, including fast spherical harmonic transforms on the
sphere and nonequispaced Fourier transforms on the torus; see, e.g.,
\cite{MR1277214,MR3163249,MR1999564,MR2329012,
KeinerKunisPotts2009NFFT,MR2211433}.  Let
\(\mathsf N_j:=|\Xc_j|\), \(M_j:=|\Lambda_j^\alpha|\), and let
\(\mathfrak C_j\) be an upper bound for the cost of one application of
either \(\mathbf F_j\) or \(\mathbf F_j^*\).  The estimates below assume
a fixed fast-transform accuracy and exclude one-time planning and
precomputation costs.  Diagonal filter multiplication costs \(O(M_j)\).

If \(I_{\rm fit}\) iterations are used in
Algorithm~\ref{alg:finest-level-cg}, finest-level coefficient preparation
costs
\[
        O\bigl((2I_{\rm fit}+1)\mathfrak C_K
        +I_{\rm fit}M_K\bigr).
\]
A coefficient analysis pass costs
\[
        O\!\left(
        r\sum_{j=J+1}^{K}\mathfrak C_j+\mathfrak C_J
        +(r+1)\sum_{j=J+1}^{K}M_j
        \right),
\]
and an adjoint synthesis pass has the same order.  If
\(I_{\rm rec}\) coefficient-space CG iterations are used, canonical
reconstruction costs one initial adjoint synthesis plus
\(I_{\rm rec}\) analysis--synthesis pairs and \(O(I_{\rm rec}M_K)\)
additional vector operations.  Projection of a spectral coefficient
vector onto a subspace such as \(\Pd_{2^L}\) costs \(O(M_K)\).
For fixed-tolerance dyadic measures, the constants may depend on the fixed
tolerance \(\eta\), but
$\mathsf N_j\asymp M_j\asymp2^{jd}$.
Here the estimate for \(\mathsf N_j\) follows from the nested
construction, while that for \(M_j\) follows from Weyl's law: continuity
and \(\widehat\alpha(0)\neq0\) make \(\widehat\alpha\) nonzero near the
origin, and its compact support gives the matching upper bound.

For fixed dimension and transform accuracy, the torus nonequispaced fast
Fourier transforms (NFFT) and its adjoint satisfy
\(
        \mathfrak C_j
        =
        O\!\left(M_j\log M_j+\mathsf N_j\right)
\)
\cite{KeinerKunisPotts2009NFFT}.  On \(\mathbb S^2\), letting
\(M_j\asymp n_j^2\), the nonequispaced fast spherical Fourier transform (NFSFT) and its adjoint satisfy
\(
        \mathfrak C_j
        =
        O\!\left(n_j^2\log^2 n_j+\mathsf N_j\right)
        =
        O\!\left(M_j\log^2 M_j+\mathsf N_j\right)
\)
 \cite{MR2329012}.  The diagonal weight scaling costs
\(O(\mathsf N_j)\) and does not change these estimates.  Thus both cases
are covered by
\[
        \mathfrak C_j
        =
        O\!\left(M_j\log^\sigma M_j+\mathsf N_j\right),
        \qquad
        \sigma=
        \begin{cases}
        1,&\Mc=\mathbb T^d,\\
        2,&\Mc=\mathbb S^2.
        \end{cases}
\]
Since \(M_j\asymp\mathsf N_j\) and both grow geometrically, one analysis
or adjoint synthesis pass costs
\[
        O\!\left((r+1)\mathsf N_K
        \log^\sigma\mathsf N_K\right).
\]
Hence \(I_{\rm rec}\) coefficient-space CG iterations cost this
quantity times \(I_{\rm rec}\), up to the initial adjoint synthesis,
and the coefficient storage is \(O((r+1)\mathsf N_K)\).
For an oversampling shift \(s\),
\(\mathsf N_{j+s}\asymp2^{sd}\mathsf N_j\), while \(M_j\) is unchanged.
Therefore
\[
        \mathfrak C_j^{(s)}
        =
        O\!\left(M_j\log^\sigma M_j+\mathsf N_{j+s}\right),
\]
and one oversampled analysis or adjoint synthesis pass costs
\[
        O\!\left((r+1)
        \bigl(M_K\log^\sigma M_K+\mathsf N_{K+s}\bigr)\right).
\]
The corresponding CG cost is this quantity times \(I_{\rm rec}\), up
to the initial adjoint synthesis, and the storage is
\(O((r+1)\mathsf N_{K+s})\).  Thus oversampling increases the
node-dependent cost and storage by a factor of order \(2^{sd}\).

\section{Approximation error}
\label{sec:approximation-error}

The fully discrete transform reconstructs an arbitrary
vector in the terminal coefficient space.  For function approximation,
we instead prescribe a target polynomial space \(\Pd_{2^L}\) and
restrict the reconstruction to that space. Let
\[
        \tau_\alpha
        :=
        \max\{\tau>0:\widehat\alpha(\xi)=1
        \text{ for }0\leq\xi\leq\tau\}>0
\]
and 
\[q_\alpha
        :=
        \min\{q\in\mathbb N_0:2^{-q}\leq\tau_\alpha\}.\]
We analyze \(\Pd_{2^L}\) from the terminal level
\(K_\alpha:=L+q_\alpha\), where we assume \(K_\alpha>J\).  Then
\begin{equation}
\label{eq:terminal-identity-on-polynomials}
        \widehat\alpha\left({\lambda_\ell}/{2^{K_\alpha}}\right)=1,
        \qquad \lambda_\ell\leq2^L.
\end{equation}
Consequently, the terminal scaling coefficients of
\(P\in\Pd_{2^L}\) are its ordinary Fourier coefficients, extended by
zero to \(\Lambda_{K_\alpha}^\alpha\).  We denote this coefficient vector by
\(\widehat{\mathbf P}^{K_\alpha}\), and define the restricted analysis map
\begin{equation}
\label{eq:polynomial-coefficient-analysis}
        \mathcal B_{J,L}P
        :=
        \mathcal A_{J,K_\alpha}\widehat{\mathbf P}^{K_\alpha},
        \qquad P\in\Pd_{2^L}.
\end{equation}
Its adjoint is the one-pass synthesis followed by spectral truncation:
\[
        \mathcal B_{J,L}^*\mathbf c
        =\sum_{\lambda_\ell\leq2^L}
        \bigl(\mathcal A_{J,K_\alpha}^*\mathbf c\bigr)_\ell u_\ell.
\]
The canonical reconstruction can be computed by
Algorithm~\ref{alg:canonical-cg-reconstruction} with
\(\mathcal B_{J,L}\) in place of \(\mathcal A_{J,K_\alpha}\).
In each application of \(\mathcal B_{J,L}^*\), the coefficients with
\(\lambda_\ell>2^L\) are discarded.  This projection is part of the
restricted adjoint and ensures that the CG iteration remains in
\(\mathbb P_{2^L}\).

\begin{theorem}[Polynomial reconstruction and approximation]
\label{thm:polynomial-reconstruction-approximation}
Let $J_0\in\mathbb{N}_0$ be an integer and $J\geq J_0$.
Let \(K_\alpha=L+q_\alpha>J\), and suppose that \(0\leq\varepsilon<1\) and
\[
        (1-\varepsilon)\|P\|_{L^2(\mu)}^2
        \leq
        \|\mathcal B_{J,L}P\|_2^2
        \leq
        (1+\varepsilon)\|P\|_{L^2(\mu)}^2,
        \qquad P\in\Pd_{2^L}.
\]
For \(f\in L^2(\Mc,\mu)\), let
\[
        f_L
        :=
        \sum_{\lambda_\ell\leq2^L}\widehat f_\ell u_\ell
\]
be its \(L^2(\mu)\)-orthogonal projection onto \(\Pd_{2^L}\), and consider the data model
\begin{equation}\label{equ:datamodel}
\mathbf c=\mathcal B_{J,L}f_L+\mathbf e,
\end{equation}  
where $\mathbf e$ encodes potential processing of the framelet coefficients.
Define the one-pass
reconstruction by
\[
        \widetilde f_L
        :=\mathcal B_{J,L}^*\mathbf c
\]
and the canonical reconstruction by
\[
        u_L
        :=
        \operatorname*{argmin}_{P\in\Pd_{2^L}}
        \|\mathcal B_{J,L}P-\mathbf c\|_2^2.
\]
 Then
\begin{equation}
\label{eq:function-one-pass-approximation-error}
        \|\widetilde f_L-f\|_{L^2(\mu)}
        \leq
        \|f-f_L\|_{L^2(\mu)}
        +\varepsilon\|f_L\|_{L^2(\mu)}
        +(1+\varepsilon)^{1/2}\|\mathbf e\|_2,
\end{equation}
and
\begin{equation}
\label{eq:function-canonical-approximation-error}
        \|u_L-f\|_{L^2(\mu)}
        \leq
        \|f-f_L\|_{L^2(\mu)}
        +\frac{1}{\sqrt{1-\varepsilon}}
        \|\mathbf e\|_2.
\end{equation}
\end{theorem}

\begin{proof}
Set
\(\mathcal S_L:=\mathcal B_{J,L}^*\mathcal B_{J,L}\).
The assumed bounds imply
\[
        (1-\varepsilon)I
        \leq\mathcal S_L\leq
        (1+\varepsilon)I.
\]
Hence \(\|\mathcal S_L-I\|\leq\varepsilon\) and
 $\|\mathcal B_{J,L}^*\|
 =\|\mathcal B_{J,L}\|
 \leq(1+\varepsilon)^{1/2}$.
For the one-pass reconstruction, the data model
\eqref{equ:datamodel} gives
\[
\widetilde f_L
 =\mathcal B_{J,L}^*\mathbf c
 =\mathcal S_Lf_L+\mathcal B_{J,L}^*\mathbf e.\]
Consequently, $\widetilde f_L-f_L
 =(\mathcal S_L-I)f_L+\mathcal B_{J,L}^*\mathbf e$,
and therefore
\[
\|\widetilde f_L-f_L\|_{L^2(\mu)}
\leq
\varepsilon\|f_L\|_{L^2(\mu)}
+(1+\varepsilon)^{1/2}\|\mathbf e\|_2.
\]
Combining this estimate with
$\|\widetilde f_L-f\|_{L^2(\mu)}
 \leq
 \|\widetilde f_L-f_L\|_{L^2(\mu)}
 +\|f_L-f\|_{L^2(\mu)}$
proves \eqref{eq:function-one-pass-approximation-error}.
For the canonical reconstruction, its normal equation
$ (\mathcal B_{J,L}^*\mathcal B_{J,L})u_L
        =\mathcal B_{J,L}^*\mathbf c$
and the data
model \eqref{equ:datamodel} yield
\(
 \mathcal S_Lu_L
 =\mathcal B_{J,L}^*\mathbf c
 =\mathcal S_Lf_L+\mathcal B_{J,L}^*\mathbf e,
\)
hence
\[
 \mathcal S_L(u_L-f_L)
 =\mathcal B_{J,L}^*\mathbf c-\mathcal B_{J,L}^*\mathcal B_{J,L}f_L
 =\mathcal B_{J,L}^*\mathbf e.
\]
Since \(\mathcal S_L\) is positive definite on \(\mathbb P_{2^L}\) and hence
invertible, it follows
\[
 u_L-f_L
 =\mathcal S_L^{-1}\mathcal B_{J,L}^*\mathbf e.
\]
Moreover,
\[
 \bigl(\mathcal S_L^{-1}\mathcal B_{J,L}^*\bigr)
 \bigl(\mathcal S_L^{-1}\mathcal B_{J,L}^*\bigr)^*
 =
 \mathcal S_L^{-1}\mathcal B_{J,L}^*
 \mathcal B_{J,L}\mathcal S_L^{-1}
 =
 \mathcal S_L^{-1}\mathcal S_L\mathcal S_L^{-1}
 =
 \mathcal S_L^{-1}.
\]
Hence
\(\|\mathcal S_L^{-1}\mathcal B_{J,L}^*\|
\leq(1-\varepsilon)^{-1/2}\), and therefore 
$\|u_L-f_L\|_{L^2(\mu)}
\leq
\|\mathbf e\|_2/\sqrt{1-\varepsilon}$.
The triangle inequality with \(f-f_L\) now proves
\eqref{eq:function-canonical-approximation-error}.  
\end{proof}

\begin{remark}
If \(\mathbf e=0\), meaning that the decomposed coefficients are not perturbed or otherwise modified, then \(u_L=f_L\). In particular, every
\(f\in\Pd_{2^L}\) is reconstructed exactly in this case.
\end{remark}

\begin{remark}[Finite-iteration CG error]
The function \(u_L\) above is the exact solution of the restricted
normal equation.  In computation, a conjugate-gradients method produces an
iterate \(u_L^{(m)}\) with residual
\[
        \mathbf r_m
        :=
        \mathcal B_{J,L}^*\mathbf c
        -\mathcal S_Lu_L^{(m)}=\mathcal S_L(u_L-u_L^{(m)}).
\]
Therefore, $u_L^{(m)}-u_L
        =-\mathcal S_L^{-1}\mathbf r_m$.
Using \(\|\mathcal S_L^{-1}\|\leq(1-\varepsilon)^{-1}\),
\[
        \|u_L^{(m)}-u_L\|_{L^2(\mu)}
        \leq
        \frac{\|\mathbf r_m\|_2}{1-\varepsilon}.
\]
Combining this estimate with
\eqref{eq:function-canonical-approximation-error} gives the total error
bound for the finite-iteration reconstruction
\[
        \|u_L^{(m)}-f\|_{L^2(\mu)}
        \leq
        \|f-f_L\|_{L^2(\mu)}
        +\frac{\|\mathbf e\|_2}{\sqrt{1-\varepsilon}}
        +\frac{\|\mathbf r_m\|_2}{1-\varepsilon}.
\]
Thus the total error consists of the polynomial approximation error,
the coefficient perturbation error, and the finite-iteration CG error.
Moreover, the spectral bounds imply
\(\kappa(\mathcal S_L)\leq(1+\varepsilon)/(1-\varepsilon)\), so a smaller
frame-bound deviation also improves the conditioning of the CG system.
\end{remark}

\begin{corollary}[Polynomial reconstruction for nested Marcinkiewicz--Zygmund measures]
\label{cor:polynomial-reconstruction-mz}
Let \(K_\alpha=L+q_\alpha>J\), and use the notation of
Theorem~\ref{thm:polynomial-reconstruction-approximation}.
\begin{enumerate}[label=\textup{(\roman*)}]
\item For the basic system in Definition \ref{def:semi-discrete-framelet-system}, with \(p=2\) and tolerance
\(\eta\),
Proposition~\ref{prop:finite-coefficient-frame-bounds}\textup{(i)}
gives \(\varepsilon=\eta\).  Hence
\[
        \|u_L-f\|_{L^2(\mu)}
        \leq
        \|f-f_L\|_{L^2(\mu)}
        +\frac{\|\mathbf e\|_2}{\sqrt{1-\eta}}.
\]

\item For the oversampled system in
Definition~\ref{def:oversampled-framelet-system}, let
\(\mathcal B_{J,L}^{(s)}\) denote the corresponding restricted analysis
map, and suppose that \(\rho_s=C_\Mc2^{-s}<1\).  By
Proposition~\ref{prop:finite-coefficient-frame-bounds}\textup{(ii)},
this map has \(\varepsilon=\rho_s\).  If
\(\mathbf c=\mathcal B_{J,L}^{(s)}f_L+\mathbf e\), its canonical
reconstruction \(u_L^{(s)}\) satisfies
\[
        \|u_L^{(s)}-f\|_{L^2(\mu)}
        \leq
        \|f-f_L\|_{L^2(\mu)}
        +\frac{\|\mathbf e\|_2}{\sqrt{1-\rho_s}}.
\]
\end{enumerate}
\end{corollary}

The error term $\|f-f_L\|_{L^2(\mu)}$ in these estimate can be refined if $f$ has extra smoothness.
\begin{remark}[Sobolev approximation \cite{MR2274179}]
If \(t>0\) and \(f\) belongs to the spectral Sobolev space
\[
        H^t(\Mc)
        :=\left\{f\in L^2(\Mc):
        \sum_{\ell\geq0}(1+\lambda_\ell^2)^t
        |\widehat f_\ell|^2<\infty\right\},
\]
then $\|f-f_L\|_{L^2(\mu)}
        \leq 2^{-Lt}\|f\|_{H^t(\Mc)}$.
\end{remark}

\section{Multiscale data analysis}
\label{sec:multiscale-data-analysis}

We then illustrate the nested framelet transform on the sphere
and on the torus, focusing on oversampling, multiscale decomposition,
and reconstruction.  

\subsection{Common numerical setup}

We consider the two-high-pass filter bank of
Wang and Zhuang in \cite{wang2020tight}, built from the smooth cutoff used
in \cite[Chapter~4]{MR1162107}.  Let
\[
        \vartheta(t)=
        \begin{cases}
        0, & t<0,\\
        t^4(35-84t+70t^2-20t^3), & 0\leq t\leq1,\\
        1, & t>1 .
        \end{cases}
\]
On the frequency interval \(0\leq |\xi|\leq 1/2\), the filters are
defined by
\begin{equation}
\label{eq:numerical-filter-bank}
\begin{aligned}
\widehat a(\xi)
&=
\begin{cases}
1, & |\xi|<1/8,\\
\cos\!\left(\frac{\pi}{2}\vartheta(8|\xi|-1)\right),
        & 1/8\leq |\xi|\leq1/4,\\
0, & 1/4<|\xi|\leq1/2,
\end{cases}
\\
\widehat {b^1}(\xi)
&=
\begin{cases}
0, & |\xi|<1/8,\\
\sin\!\left(\frac{\pi}{2}\vartheta(8|\xi|-1)\right),
        & 1/8\leq |\xi|\leq1/4,\\
\cos\!\left(\frac{\pi}{2}\vartheta(4|\xi|-1)\right),
        & 1/4<|\xi|\leq1/2,
\end{cases}
\\
\widehat {b^2}(\xi)
&=
\begin{cases}
0, & |\xi|<1/4,\\
\sin\!\left(\frac{\pi}{2}\vartheta(4|\xi|-1)\right),
        & 1/4\leq |\xi|\leq1/2 .
\end{cases}
\end{aligned}
\end{equation}
The associated generators obtained from the refinement equations
\eqref{eq:filter-refinement-equations} are
\begin{equation}
\label{eq:numerical-generators}
\begin{aligned}
\widehat\alpha(\xi)
&=
\begin{cases}
1, & |\xi|<1/4,\\
\cos\!\left(\frac{\pi}{2}\vartheta(4|\xi|-1)\right),
        & 1/4\leq |\xi|\leq1/2,\\
0, & \text{otherwise},
\end{cases}
\\
\widehat{\beta^1}(\xi)
&=
\begin{cases}
\sin\!\left(\frac{\pi}{2}\vartheta(4|\xi|-1)\right),
        & 1/4\leq |\xi|<1/2,\\
\cos^2\!\left(\frac{\pi}{2}\vartheta(2|\xi|-1)\right),
        & 1/2\leq |\xi|\leq1,\\
0, & \text{otherwise},
\end{cases}
\\
\widehat{\beta^2}(\xi)
&=
\begin{cases}
\cos\!\left(\frac{\pi}{2}\vartheta(2|\xi|-1)\right)
\sin\!\left(\frac{\pi}{2}\vartheta(2|\xi|-1)\right),
        & 1/2\leq |\xi|\leq1,\\
0, & \text{otherwise}.
\end{cases}
\end{aligned}
\end{equation}
Figure~\ref{fig:filter-generator} displays the filter symbols and the
associated generator functions.

\begin{figure}[htbp]
        \centering
        \includegraphics[width=\textwidth]{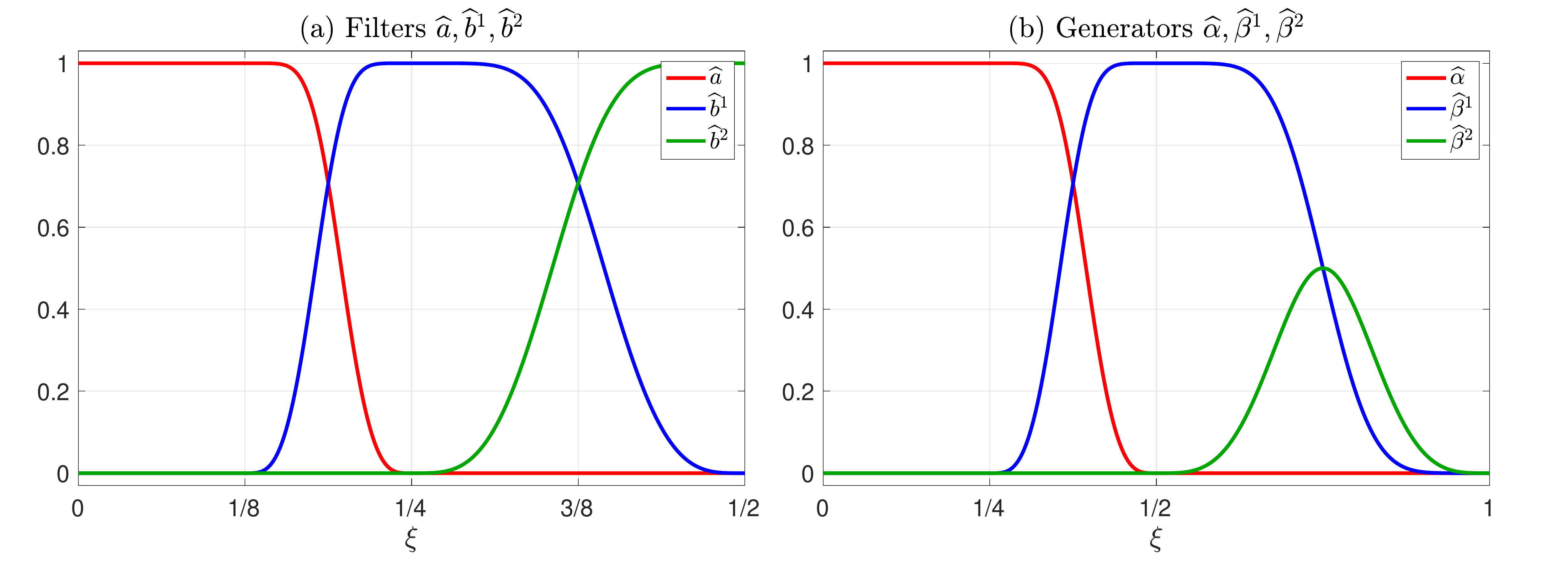}
        \caption{A two-high-pass filter bank and the associated
        framelet generators.}
        \label{fig:filter-generator}
\end{figure}

The assumptions used in the preceding sections are readily checked for
this choice.  The filters in \eqref{eq:numerical-filter-bank} satisfy
\begin{equation}
\label{eq:numerical-filter-partition}
        |\widehat a(\xi)|^2
        +|\widehat {b^1}(\xi)|^2
        +|\widehat {b^2}(\xi)|^2
        =1,\qquad 0\leq |\xi|\leq1/2 .
\end{equation}
This follows piecewise from \(\cos^2\theta+\sin^2\theta=1\), and,
together with the refinement relations
\eqref{eq:filter-refinement-equations}, gives
\eqref{eq:spectral-refinement}.  The support condition
\eqref{eq:filter-unit-support} is immediate from
\eqref{eq:numerical-generators} since
$\supp\widehat\alpha\subset[-1/2,1/2]$ and
$\supp\widehat{\beta^1}\cup\supp\widehat{\beta^2}
        \subset[-1,1]$.
Moreover, \(\widehat\alpha(\xi)=1\) for \(0\leq|\xi|\leq1/4\).  Hence the
low-pass limit condition \eqref{eq:spectral-lowpass-limit} holds and
\(\tau_\alpha=1/4\).  In the coefficient formulation of
Section~\ref{sec:fully-discrete-transform} this
gives \(q_\alpha=2\): a target polynomial space \(\Pd_{2^L}\) is analyzed
from finest scaling level \(K=L+2\).  At that level the finest-level
scaling multiplier is identically one on every frequency
\(\lambda_\ell\leq2^L\), as stated in
\eqref{eq:terminal-identity-on-polynomials}.  This finest-level offset is used
in the polynomial reconstruction and approximation results of
Section~\ref{sec:approximation-error}.
The finite numerical experiments below instead keep the prescribed finest
bandwidth fixed: they first compute a weighted least-squares projection in
that space and then apply the coefficient filter bank directly, without
adding two finest levels.

Both CG routines start from zero and terminate when the relative
residual of the corresponding normal equation satisfies
\begin{equation}
\label{eq:numerical-cg-stopping-rule}
        \frac{\|\mathbf r_m\|_2}
        {\max\{\|\mathbf b\|_2,\epsilon_{\mathrm{mach}}\}}
        \leq\tau .
\end{equation}
or when the prescribed iteration cap is reached.
Here \(H\mathbf z=\mathbf b\) is either the finest-level least-squares
normal equation, with \(H=\mathbf F^*\mathbf F\), or the canonical
frame-operator equation, with
\(H=\mathcal A^*\mathcal A\).  We use \(\tau_{\mathrm{fit}}=10^{-10}\) for every finest-level
least-squares projection.  The canonical reconstruction tolerance is
\(\tau_{\mathrm{rec}}=10^{-9}\) in the following Examples~1--3 and
\(\tau_{\mathrm{rec}}=10^{-8}\) in Example~4.  The corresponding
iteration caps \((M_{\mathrm{fit}},M_{\mathrm{rec}})\) are
\((140,50)\), \((120,40)\), \((80,40)\), and \((80,25)\),
respectively.

All numerical experiments were performed in MATLAB R2024a on a MacBook Pro (16-inch, 2019) with a 2.6 GHz 6-core Intel Core i7 processor and 16 GB RAM.  The fast transforms were evaluated with the NFFT 3.5.3 MATLAB package \cite{KeinerKunisPotts2009NFFT} via its macOS MEX interfaces; the spherical experiments use its NFSFT routines when
available, and the torus experiments use its two-dimensional NFFT
routines.

\subsection{On the sphere}

For the spherical experiments, at a dyadic level \(j\) we take
\(n_j=2^j\) as the spherical harmonic degree cutoff and choose
$|\Xc_j|=\lceil m(2^j+1)^2\rceil
        =\lceil m(n_j+1)^2\rceil$,
where \(m\geq1\) is the oversampling factor.  
The nodes are generated from a large deterministic Fibonacci candidate
set by a greedy farthest-point ordering with chordal distance; the sets
\(\Xc_j\) are nested prefixes of this ordering with the cardinalities
specified above. Given the nested node set
\(\Xc_j\), we use the surface-area version of the nearest-neighbor
partition weights in Definition~\ref{def:partition-weights}.
Numerically, a dense Fibonacci proxy set
\(\{z_q\}_{q=1}^Q\) is assigned to the nearest sample points and we set
\[
        w_{j,x}
        =
        \frac{4\pi}{Q}
        \#\{q:x\text{ is nearest to }z_q\}.
\]
On the unit sphere, nearest neighbors in chordal and geodesic distance
coincide, so these weights approximate the areas of the corresponding
spherical nearest-neighbor cells.
Here we use the standard surface-area normalization
\(\sum_{x\in\Xc_j}w_{j,x}=4\pi\); this differs from the probability
normalization in the theoretical sections only by a constant factor.
These farthest-point prefixes with partition-type weights are not
spherical designs or exact quadrature rules; they are used here as
simple quasi-uniform nested Marcinkiewicz--Zygmund-type sampling rules.

\paragraph{Example 1: Oversampling for a Wendland test function.}
We first test the effect of oversampling on reconstruction.  The signal
is a sum of Wendland-type bumps \cite{chernih2014wendland} on
\(\mathbb S^2\):
\[
        f(x)=\sum_{c\in\{\pm e_1,\pm e_2,\pm e_3\}}
        \phi(\|x-c\|_2),
\]
where
\[
        \phi(r)=
        \frac{(1-r)_{+}^{10}
        \left(429r^4+450r^3+210r^2+50r+5\right)}
        {5\tau_4},
        \qquad
        \tau_4=\frac{15\,\Gamma(9/2)}{2\,\Gamma(5)} .
\]
Here \(e_1,e_2,e_3\) are the coordinate vectors in \(\mathbb R^3\), and
\((t)_+=\max\{t,0\}\).  We use two finest polynomial degrees,
\(n=16\) and \(n=32\), corresponding respectively to the dyadic levels
\(j=4,3,2\) and \(j=5,4,3\).  For each fixed degree, we vary the
oversampling factor over the twelve values $m_i=1.2+0.1i(i+1)$ for $i=1,\ldots,12$. 

Figure~\ref{fig:wendland-reconstruction} shows the Wendland signal, the
one-pass reconstruction and its pointwise error, and the canonical CG
reconstruction and its pointwise error, for the case \(n=32\) with
oversampling factor \(m=1.4\).  The pointwise errors are measured
against the bandlimited projection on an independent Fibonacci test
grid.

\begin{figure}[htbp]
        \centering
        \begin{minipage}{0.3\textwidth}
        \centering
        \textbf{(a)} Wendland function\\ 
        \includegraphics[width=\linewidth]{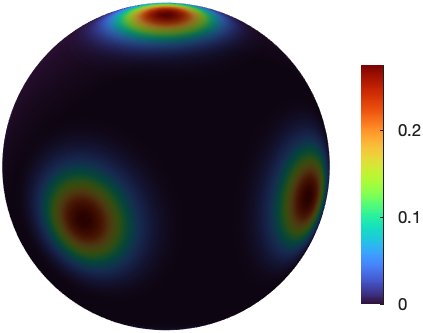}
        \end{minipage}
        \begin{minipage}{0.3\textwidth}
        \centering
        \textbf{(b)} One-pass reconstruction\\
        \includegraphics[width=\linewidth]{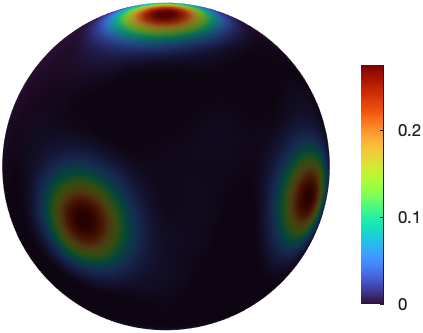}
        \end{minipage}
        \begin{minipage}{0.3\textwidth}
        \centering
        \textbf{(c)} CG reconstruction\\
        \includegraphics[width=\linewidth]{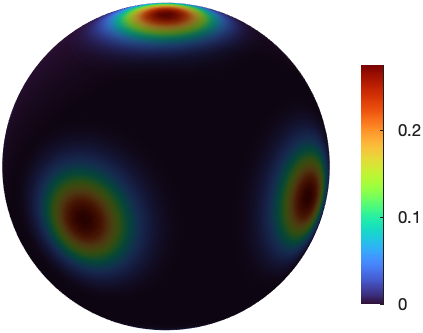}
        \end{minipage}

                \vspace{2ex}

                \begin{minipage}{0.33\textwidth}
        \centering
        \textbf{(d)} One-pass error\\
        \includegraphics[width=\linewidth]{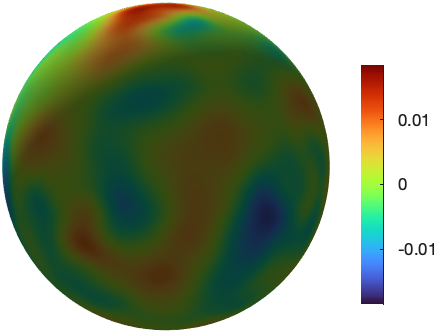}
        \end{minipage}
        \begin{minipage}{0.33\textwidth}
        \centering
        \textbf{(e)} CG error\\
        \includegraphics[width=\linewidth]{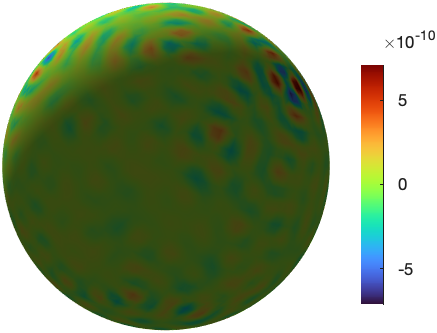}
        \end{minipage}
        \caption{Wendland-type signal reconstructions on
        \(\mathbb S^2\), for \(n=32\) and oversampling factor \(m=1.4\).}
        \label{fig:wendland-reconstruction}
\end{figure}

Figure~\ref{fig:wendland-oversampling} reports the corresponding
relative errors as the number of finest-level nodes increases.  The
one-pass synthesis error decreases as the sampling density increases,
while the canonical CG reconstruction error is several orders of
magnitude smaller.  Its eventual scale is compatible with the
reconstruction residual tolerance \(\tau_{\mathrm{rec}}=10^{-9}\),
after allowing for the conditioning and the fast-transform error.

\begin{figure}[htbp]
        \centering
        \includegraphics[width=\textwidth]{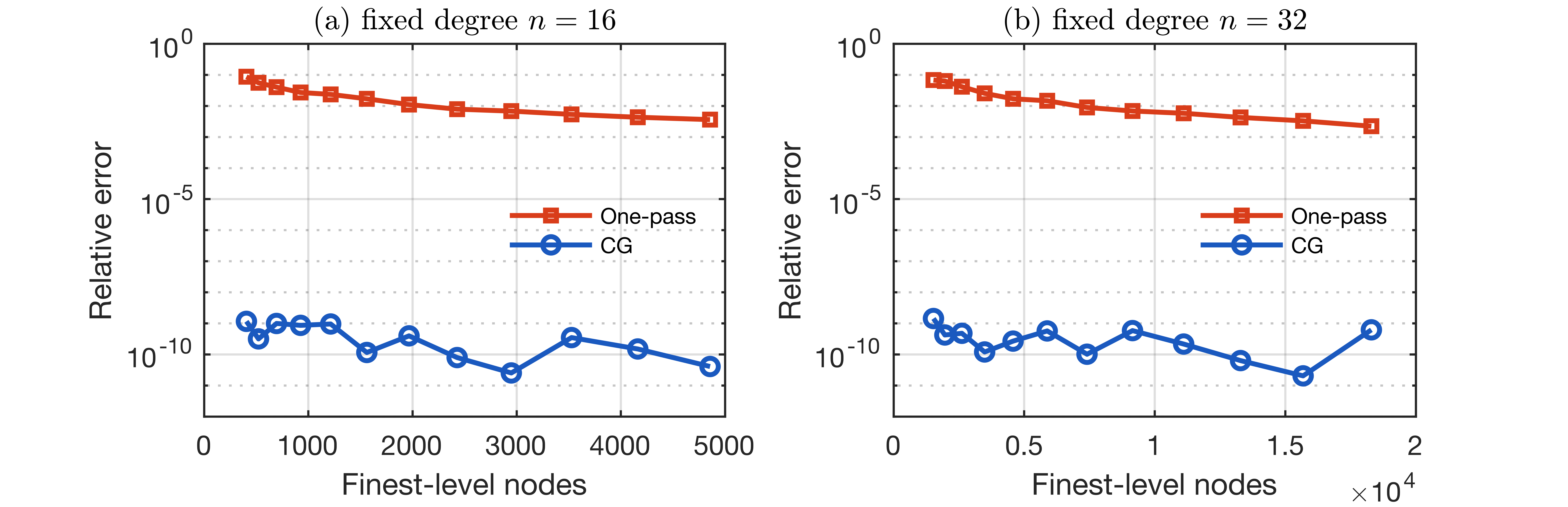}
        \caption{Relative reconstruction error for the Wendland-type
        test function as the number of finest-level nodes increases. }
        \label{fig:wendland-oversampling}
\end{figure}

\paragraph{Example 2: Multiscale analysis of an Earth texture signal.}
We next apply the same transform to an Earth texture signal.  The input
is obtained from a Natural Earth texture map\footnote{
Adapted from Natural Earth III by Tom Patterson
(\url{https://www.shadedrelief.com}); public domain.} by converting the
image to grayscale and sampling it at the finest nested nodes.  The
experiment therefore treats image intensity as a scalar function on
\(\mathbb S^2\), not as physical elevation data. We use levels \(j=7,6,5\), with spherical harmonic degree cutoffs
$n_7=128$, $n_6=64$, and $n_5=32$.
With oversampling factor \(m=2.5\), the nested farthest-point prefixes
have cardinalities
$|\Xc_7|=41603$, $|\Xc_6|=10563$, and $|\Xc_5|=2723$.
The ordering is selected from \(249618\) Fibonacci candidates; a
separate Fibonacci proxy set with \(499236\) points is used only to
approximate the nearest-neighbor cell weights.

Let \(\mathbf v\) denote the sampled texture data on \(\Xc_7\).  We
compute its weighted least-squares projection \(\mathbf v_7\) onto
spherical harmonics of degree at most \(128\), using conjugate gradients
with tolerance \(10^{-10}\) and at most \(120\) iterations.  The residual
is \(\mathbf w_7=\mathbf v-\mathbf v_7\).  Starting from the finest
spectral coefficient vector, two filter-bank analysis steps produce
\(\mathbf v_6,\mathbf w_6^1,\mathbf w_6^2\) and then
\(\mathbf v_5,\mathbf w_5^1,\mathbf w_5^2\).  The projection relative
error is \(1.5885\times10^{-1}\).  Relative to the projected signal
\(\mathbf v_7\), the one-pass synthesis error is
\(3.4203\times10^{-2}\), while the canonical coefficient-space CG
reconstruction reduces this error to \(6.4390\times10^{-10}\).  The
corresponding relative error against the original image data remains
\(1.5885\times10^{-1}\), with SNR \(15.888\) dB. These
values are consistent with \eqref{eq:numerical-cg-stopping-rule}: the
canonical error relative to the projection is at the algebraic solver
scale, whereas the much larger error relative to the image is the
bandlimited projection error and is unaffected by further CG
iterations.  These
quantities are shown in Figure~\ref{fig:texturemap}: the approximation
terms retain the large-scale texture, while the detail terms capture
oscillatory features at the corresponding dyadic resolutions.

\begin{figure}[htbp]
        \centering
        \begin{minipage}{0.32\textwidth}
        \centering
        \textbf{(a)} Texture map data \(\mathbf v\)\\
        \includegraphics[width=\linewidth]{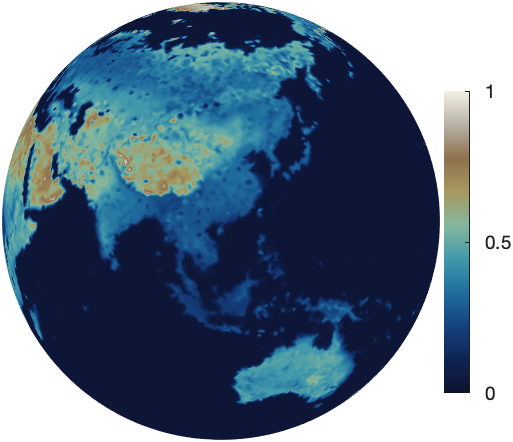}
        \end{minipage}\hfill
        \begin{minipage}{0.32\textwidth}
        \centering
        \textbf{(b)} Projection term \(\mathbf v_7\)\\
        \includegraphics[width=\linewidth]{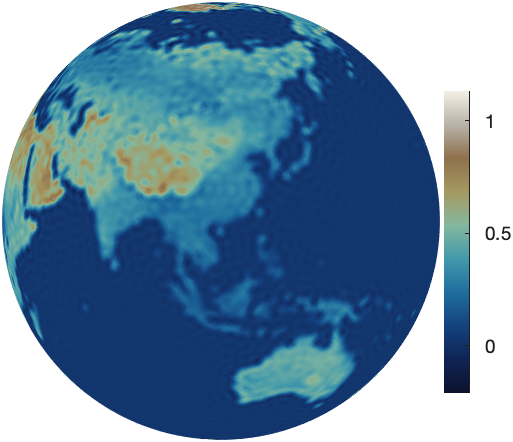}
        \end{minipage}\hfill
        \begin{minipage}{0.32\textwidth}
        \centering
        \textbf{(c)} Error term \(\mathbf w_7\)\\
        \includegraphics[width=\linewidth]{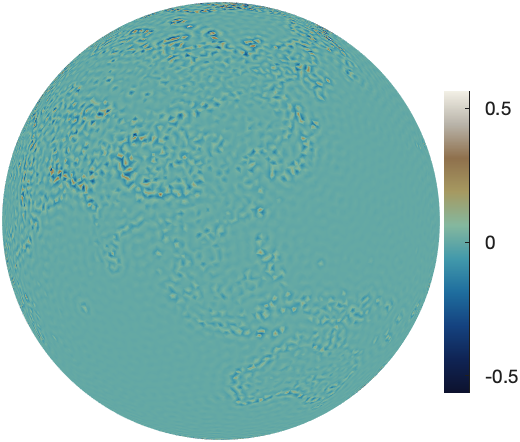}
        \end{minipage}

        \vspace{1ex}

        \begin{minipage}{0.32\textwidth}
        \centering
        \textbf{(d)} Approximation \(\mathbf v_{6}\)\\
        \includegraphics[width=\linewidth]{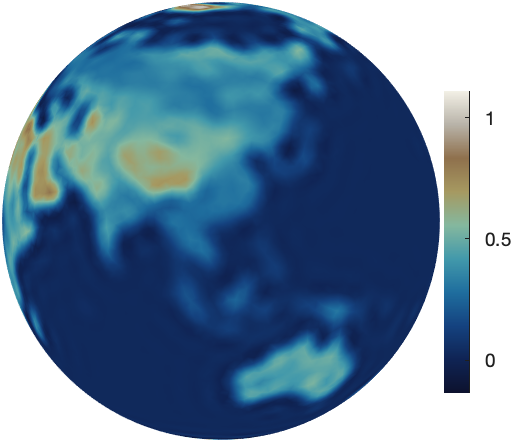}
        \end{minipage}\hfill
        \begin{minipage}{0.32\textwidth}
        \centering
        \textbf{(e)} Detail \(\mathbf w_{6}^{1}\)\\
        \includegraphics[width=\linewidth]{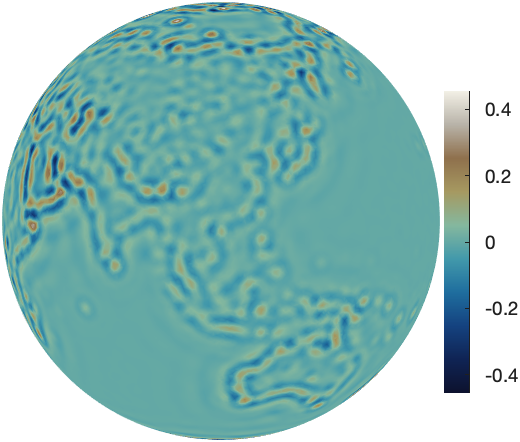}
        \end{minipage}\hfill
        \begin{minipage}{0.32\textwidth}
        \centering
        \textbf{(f)} Detail \(\mathbf w_{6}^{2}\)\\
        \includegraphics[width=\linewidth]{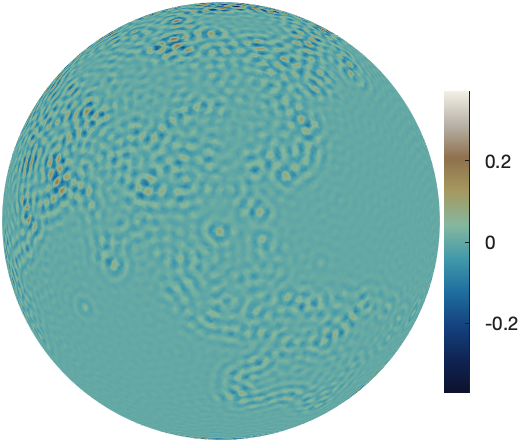}
        \end{minipage}

        \vspace{1ex}

        \begin{minipage}{0.32\textwidth}
        \centering
        \textbf{(g)} Approximation \(\mathbf v_{5}\)\\
        \includegraphics[width=\linewidth]{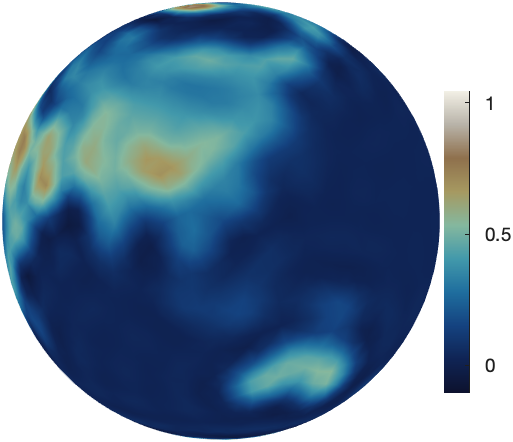}
        \end{minipage}\hfill
        \begin{minipage}{0.32\textwidth}
        \centering
        \textbf{(h)} Detail \(\mathbf w_{5}^{1}\)\\
        \includegraphics[width=\linewidth]{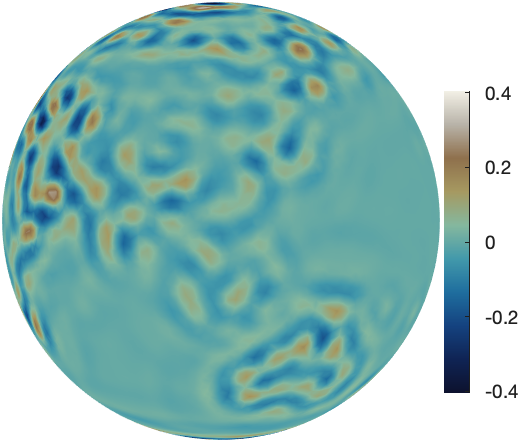}
        \end{minipage}\hfill
        \begin{minipage}{0.32\textwidth}
        \centering
        \textbf{(i)} Detail \(\mathbf w_{5}^{2}\)\\
        \includegraphics[width=\linewidth]{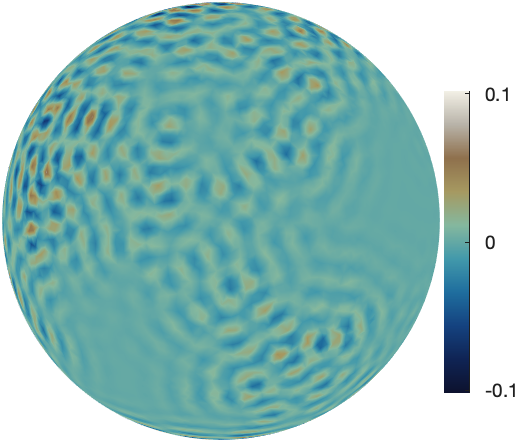}
        \end{minipage}
        \caption{Multiscale decomposition of an Earth texture signal on
        nested spherical sampling sets. }
        \label{fig:texturemap}
\end{figure}

\subsection{On the torus}

We also consider the flat torus \(\mathbb T^2=[0,1)^2\), with periodic
boundary conditions and normalized Lebesgue measure.  With the
\(2\pi\)-Fourier convention, \(u_k(x)=e^{2\pi i k\cdot x}\),
\(k\in\mathbb Z^2\), we have
$ -\Delta u_k=(2\pi)^2|k|^2u_k=\lambda_k^2u_k$ with $\lambda_k=2\pi|k|$,
        and we write
$\mathbb P_n(\mathbb T^2)
        =
        \operatorname{span}\{e^{2\pi i k\cdot x}:|k|\leq n\}$,
        where $
        D_n=\#\{k\in\mathbb Z^2:|k|\leq n\}$. 
For dyadic degrees \(n_j=2^j\), we measure oversampling relative to
\(D_{n_j}\).  Given \(m\geq1\), set
\(q_J=\lceil\sqrt{mD_{n_J}}\rceil\) and
\(q_j=2^{j-J}q_J\), so that \(|\Xc_j|=q_j^2\) has the same order as
\(mD_{n_j}\).  The set \(\Xc_j\) is a \(q_j\times q_j\) periodic
tensor-cell rule with one jittered point per cell: each grid center is
randomly displaced by a fixed fraction of the cell width and then
reduced modulo one.  The finest jittered array is generated first, and
coarser levels are dyadic subarrays, so
\(\Xc_J\subset\cdots\subset\Xc_L\).  We use the cell-area weights
\(w_{j,x}=q_j^{-2}\).  These are not exact tensor-product quadrature
rules, but simple quasi-uniform nested Marcinkiewicz--Zygmund-type sampling rules.  For
visualization, the flat torus data are displayed on the standard
embedded torus in \(\mathbb R^3\).

\paragraph{Example 3: Multiscale analysis on an oscillatory signal.}
We use dyadic levels \(j=6,5,4\), with $n_6=64$, $n_5=32$, and $n_4=16$.
The oversampling factor is \(m=2.5\), the jitter parameter is \(0.25\),
and the random seed is fixed at \(5\).  This gives $q_6=180$, $q_5=90$, and $q_4=45$,
and hence
$|\Xc_6|=32400$, $|\Xc_5|=8100$, and $|\Xc_4|=2025$.
The test signal is a synthetic periodic function consisting of
low-frequency trigonometric modes, localized periodic Gaussian bumps,
and two high-frequency oscillatory modes.  More precisely, before
normalization to the range \([0,1]\), we use
\[
\begin{aligned}
f(u,v)
={}&0.32\cos(2\pi u)+0.22\sin(2\pi v)
      +0.18\cos(2\pi(u-v)) +\sum_{\ell=1}^4 A_\ell
  \exp\!\left(-\frac{d_{\mathbb T^2}((u,v),c_\ell)^2}
                   {2\sigma_\ell^2}\right)  \\
&+0.16\sin(2\pi(24u+16v))
 +0.12\cos(2\pi(31u-11v)),
\end{aligned}
\]
where $(A_1,A_2,A_3,A_4)=(1.00,0.70,0.85,0.55)$,
        $c_1=(0.22,0.28)$, $c_2=(0.62,0.36)$,
        $c_3=(0.48,0.76)$, $c_4=(0.82,0.68)$,
and $(\sigma_1,\sigma_2,\sigma_3,\sigma_4)
        =(0.055,0.075,0.060,0.045)$.
The finest weighted least-squares projection uses conjugate gradients
with tolerance \(10^{-10}\).  The resulting spectral coefficient vector is
used as the input to the multilevel analysis transform.  In this experiment
the projection relative error is \(4.8773\times10^{-11}\).  The one-pass
synthesis differs from the projected signal by \(1.2697\times10^{-2}\), while
the canonical coefficient-space CG reconstruction reduces this error to
\(2.7717\times10^{-10}\).  The latter is below the prescribed residual
tolerance \(10^{-9}\), which is possible because the stopping criterion
is an upper threshold on the normal-equation residual rather than an
identity for the reconstructed function error.
These numbers describe the observed behavior of
the implemented jittered Marcinkiewicz--Zygmund-type rule and should not be interpreted as
computed theoretical frame bounds.  Figure~\ref{fig:torus-multiscale}
shows the resulting multiscale decomposition.  The approximation terms
retain the low-frequency trend and localized bump structure, whereas
the high-pass terms mainly capture the imposed oscillatory components
at the corresponding dyadic scales.  The flat torus data are displayed
on the standard embedded torus in \(\mathbb R^3\).  

\begin{figure}[htbp]
        \centering
        \begin{minipage}{0.32\textwidth}
        \centering
        \textbf{(a)} Torus data \(\mathbf v\)\\
        \includegraphics[width=\linewidth]{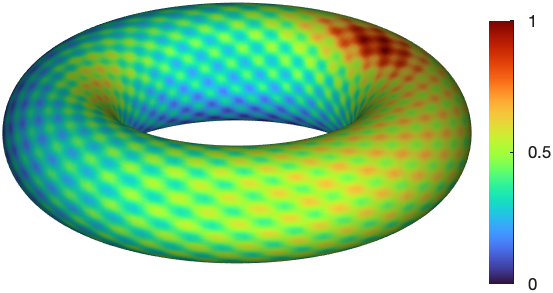}
        \end{minipage}\hfill
        \begin{minipage}{0.32\textwidth}
        \centering
        \textbf{(b)} Projection term \(\mathbf v_6\)\\
        \includegraphics[width=\linewidth]{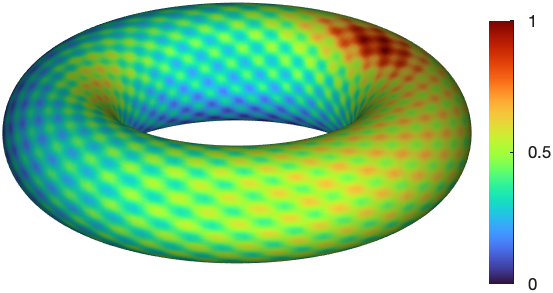}
        \end{minipage}\hfill
        \begin{minipage}{0.32\textwidth}
        \centering
        \textbf{(c)} Error term \(\mathbf w_6\)\\
        \includegraphics[width=\linewidth]{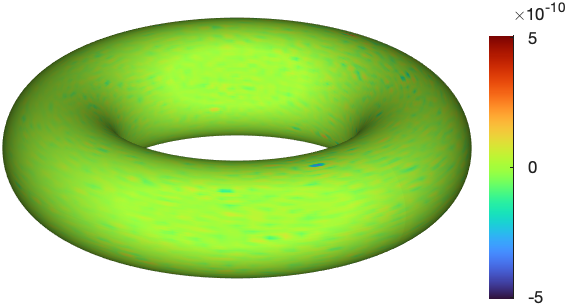}
        \end{minipage}

        \vspace{1ex}

        \begin{minipage}{0.32\textwidth}
        \centering
        \textbf{(d)} Approximation \(\mathbf v_{5}\)\\
        \includegraphics[width=\linewidth]{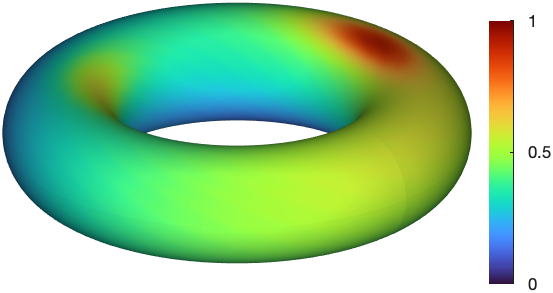}
        \end{minipage}\hfill
        \begin{minipage}{0.32\textwidth}
        \centering
        \textbf{(e)} Detail \(\mathbf w_{5}^{1}\)\\
        \includegraphics[width=\linewidth]{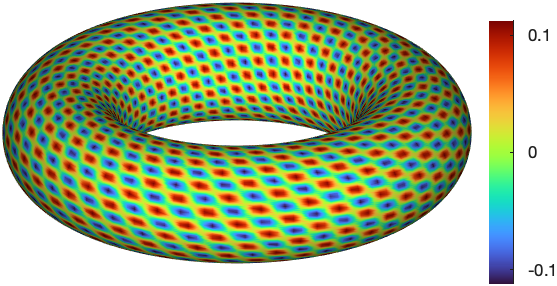}
        \end{minipage}\hfill
        \begin{minipage}{0.32\textwidth}
        \centering
        \textbf{(f)} Detail \(\mathbf w_{5}^{2}\)\\
        \includegraphics[width=\linewidth]{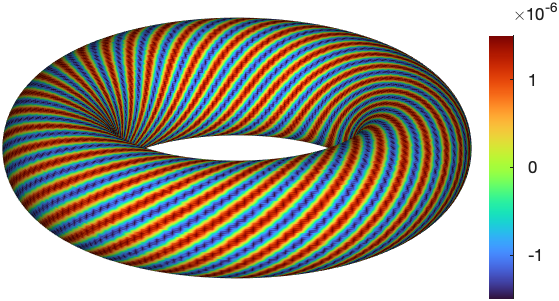}
        \end{minipage}

        \vspace{1ex}

        \begin{minipage}{0.32\textwidth}
        \centering
        \textbf{(g)} Approximation \(\mathbf v_{4}\)\\
        \includegraphics[width=\linewidth]{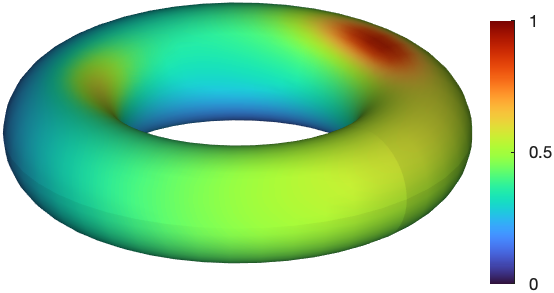}
        \end{minipage}\hfill
        \begin{minipage}{0.32\textwidth}
        \centering
        \textbf{(h)} Detail \(\mathbf w_{4}^{1}\)\\
        \includegraphics[width=\linewidth]{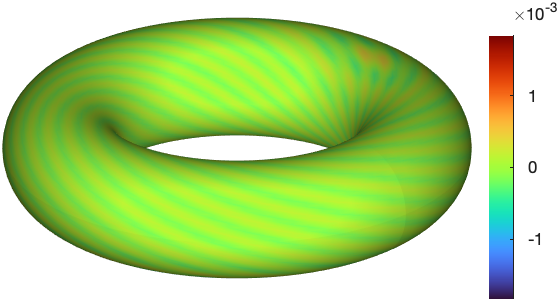}
        \end{minipage}\hfill
        \begin{minipage}{0.32\textwidth}
        \centering
        \textbf{(i)} Detail \(\mathbf w_{4}^{2}\)\\
        \includegraphics[width=\linewidth]{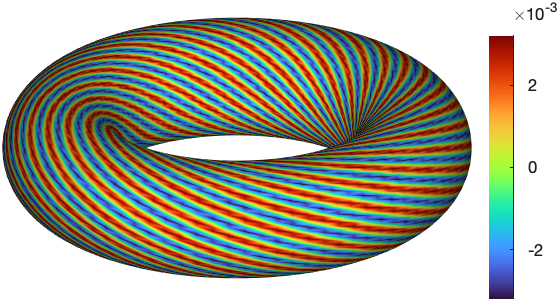}
        \end{minipage}
        \caption{Multiscale decomposition of a synthetic signal on
        jittered nested sampling sets on \(\mathbb T^2\).}
        \label{fig:torus-multiscale}
\end{figure}

\paragraph{Example 4: Denoising on an oscillatory signal.}
We next use the same jittered nested torus rule to test a simple
framelet thresholding procedure.  The levels, oversampling factor,
jitter parameter, and NFFT implementation are the same as in Example~3.
For this experiment the clean signal is the smoother version of the
periodic synthetic function, obtained by retaining the low-frequency
trigonometric modes and localized periodic Gaussian bumps and replacing
the high-frequency oscillatory terms by \(0.025\sin(2\pi(5u+3v))\).
We add Gaussian noise $g=f+\sigma\,\operatorname{std}(f)\,\epsilon$,
where the entries of \(\epsilon\) are independent standard normal
random variables.  After computing the finest weighted least-squares
projection, we apply the multilevel analysis transform to the recovered
finest spectral coefficient vector, use soft thresholding with threshold
$\eta_{\mathrm{thr}}
        =0.2\,\sigma\,\operatorname{std}(f)$,
and compare two reconstruction procedures: direct one-pass synthesis
and canonical coefficient-space CG reconstruction.  Errors in Table~\ref{tab:torus-denoising} are dominated by the
added noise and coefficient thresholding; the reconstruction tolerance
\(\tau_{\mathrm{rec}}=10^{-8}\) is several orders of magnitude smaller
and has no visible effect at the reported precision.

\begin{table}[htbp]
\centering
\small
\setlength{\tabcolsep}{5pt}
\begin{tabular}{ccccccc}
\toprule
\(\sigma\) 
& noisy rel. err. & one-pass rel. err. & canonical rel. err. & noisy SNR & one-pass SNR & CG SNR \\
\midrule
0.10  & \(4.5288\times10^{-2}\)
     & \(2.0487\times10^{-2}\) & \(1.7673\times10^{-2}\)
     & 26.880 & 33.770 & 35.054 \\
0.20 & \(9.0461\times10^{-2}\)
     & \(3.7380\times10^{-2}\) & \(3.5799\times10^{-2}\)
     & 20.871 & 28.547 & 28.922 \\
0.30  & \(1.3592\times10^{-1}\)
     & \(5.4773\times10^{-2}\) & \(5.3544\times10^{-2}\)
     & 17.334 & 25.229 & 25.426 \\
\bottomrule
\end{tabular}
\caption{Denoising results on the jittered torus rule.  SNR is reported
in dB.}
\label{tab:torus-denoising}
\end{table}

Figure~\ref{fig:torus-denoising} shows the denoising result for
\(\sigma=0.20\).  For this noise level, one-pass synthesis already
reduces the error of the noisy observation, but the canonical
coefficient-space CG reconstruction gives a smaller relative error and a
higher SNR.  This is consistent with the fact that one-pass synthesis
also contains the finite coefficient-frame discretization error, whereas the
canonical reconstruction compensates for it.

\begin{figure}[htbp]
        \centering
        \begin{minipage}{0.32\textwidth}
        \centering
        \textbf{(a)} Clean signal\\
        \includegraphics[width=\linewidth]{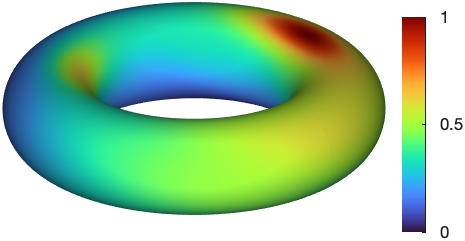}
        \end{minipage}\hfill
        \begin{minipage}{0.32\textwidth}
        \centering
        \textbf{(b)} One-pass denoising\\
        \includegraphics[width=\linewidth]{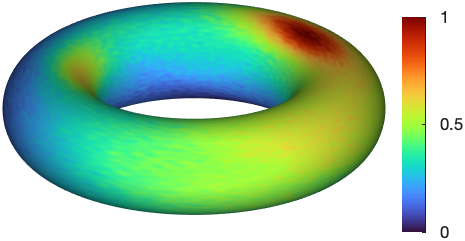}
        \end{minipage}\hfill
        \begin{minipage}{0.32\textwidth}
        \centering
        \textbf{(c)} CG denoising\\
        \includegraphics[width=\linewidth]{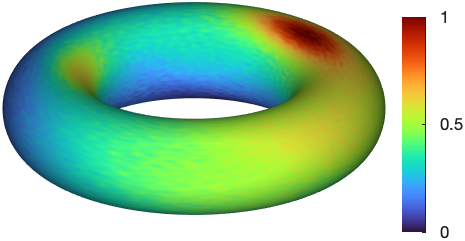}
        \end{minipage}

        \vspace{1ex}

        \begin{minipage}{0.32\textwidth}
        \centering
        \textbf{(d)} Noisy signal\\
        \includegraphics[width=\linewidth]{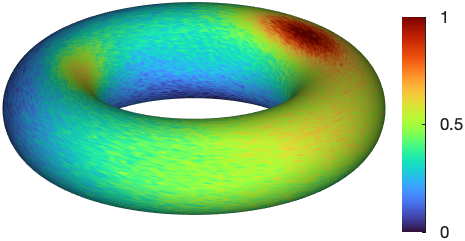}
        \end{minipage}\hfill
        \begin{minipage}{0.32\textwidth}
        \centering
        \textbf{(e)} One-pass error\\
        \includegraphics[width=\linewidth]{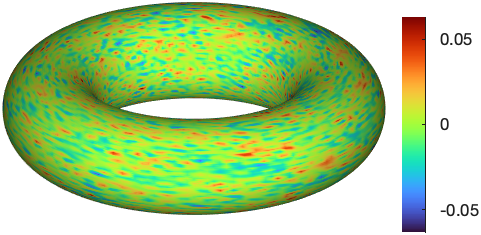}
        \end{minipage}\hfill
        \begin{minipage}{0.32\textwidth}
        \centering
        \textbf{(f)} CG error\\
        \includegraphics[width=\linewidth]{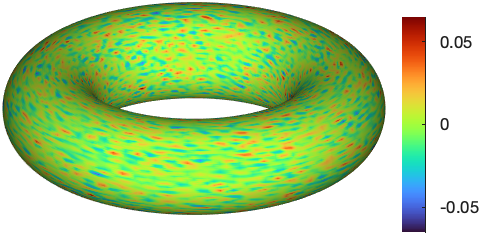}
        \end{minipage}
        \caption{Framelet denoising on the jittered torus rule for
        noise level \(\sigma=0.20\). }
        \label{fig:torus-denoising}
\end{figure}

\section{Concluding remarks}

Exact quadrature plays two logically distinct roles in many framelet constructions on manifolds.  It provides a stable discretization of bandlimited functions, but it also preserves the relevant inner products exactly, leading to tightness and one-pass perfect reconstruction.  The present work separates these two roles. Marcinkiewicz--Zygmund inequalities retain the stable discretization property but without requiring polynomial exactness.  This relaxation allows us to work with geometrically controlled scattered nodes and, in particular, to construct nested measures that retain all previous nodes as the resolution is refined.

Applying the nested Marcinkiewicz--Zygmund measures to continuous tight framelets produces nearly tight semi-discrete systems with explicitly controlled frame operators.  For a prescribed tolerance \(\eta\), the frame operator lies within \(\eta\) of the identity. Alternatively, nested oversampling reduces this deviation at the rate \(O(2^{-s})\). Therefore, relaxing polynomial exactness does not lead to an uncontrolled loss of tightness.  It makes scattered and progressively refined sampling possible while retaining energy preservation and
reconstruction stability to a quantifiable degree.

At the fully discrete level, the same near-tightness parameter controls both the error of one-pass synthesis and the conditioning of canonical CG reconstruction.  The resulting framework thus connects the geometry of nested scattered samples, Marcinkiewicz--Zygmund discretization, nearly tight
framelets, and practical multiscale transforms on manifolds.

\small
\bibliographystyle{siamplain}
\bibliography{myref}

\end{document}